\documentclass[3p]{elsarticle}
\usepackage{amsmath,amsthm}
\usepackage{amssymb}
\usepackage{amsbsy}
\usepackage{mathrsfs}
\usepackage{tikz}
\usepackage{pgfplots}
\usepackage{cases}
\usepackage{hyperref} 
\usepackage{comment}
\usepackage{ulem}
\usepackage{subcaption}
\usepackage{booktabs}
\usepackage{soul,cancel}

\newtheorem{theorem}{Theorem}
\newtheorem{lemma}[theorem]{Lemma}
\newdefinition{remark}{Remark}
\newdefinition{hyp}{Hypothesis}
\newdefinition{ass}{Assumption}
\newtheorem{algorithm}{Algorithm}[section]

\def\nuplus{\nu^{+}}
\def\numinus{\nu^{-}}
\def\pplus{p^{+}}
\def\pminus{p^{-}}
\def\fplus{\mathbf{f}^{+}}
\def\fminus{\mathbf{f}^{-}}
\def\phi{{\varphi}}
\def\V#1{\mathbf{#1}}
\def\Vt#1{\mathbf{\tilde{#1}}}
\def\VX{\mathbf{X}}
\newcommand{\Mbd}{\mathcal{M}}
\newcommand{\Nbd}{\mathcal{N}}
\newcommand{\divergence}{\mathrm{\nabla \cdot}}

\newcommand{\Lumap}{\mathrm{T_1}}
\newcommand{\Luinvmap}{\mathrm{T_1^{-1}}}
\newcommand{\Lbmap}{\mathrm{T_2}}
\newcommand{\Lbinvmap}{\mathrm{T_2^{-1}}}

\def\prodL2#1{\left(#1\right)}

\newcommand{\normBoch}[1]{{\left\vert\kern-0.25ex\left\vert\kern-0.25ex\left\vert #1 
\right\vert\kern-0.25ex\right\vert\kern-0.25ex\right\vert}}

\newcommand{\DY}[1]{{\color{black}{#1}}}
\newcommand{\AT}[1]{{\color{black}{#1}}}

\journal{Communications in Nonlinear Science and Numerical Simulation}

\begin{document}
\begin{frontmatter}

\title{Decoupled iterative schemes for solving stationary MHD problems} 

\author[labelA]{Aziz Takhirov} 
\ead{atakhirov@sharjah.ac.ae}
\address[labelA]{University of Sharjah, UAE}
\author[labelB]{Driss Yakoubi}
\ead{driss.yakoubi@devinci.fr}
\address[labelB]{De Vinci Higher Education, De Vinci Research Center, Paris, France.}


\begin{abstract}
We develop a novel iterative approach for solving the incompressible magnetohydrodynamics problem. The main idea is to split the velocity-momentum and magnetic induction equations with respect to the diffusive terms, as in \cite{Yak2023}. As a result, we get a smaller system that is iteration-level-dependent, along with two Stokes systems that need to be assembled only once. We also extended the scheme to the Els{\"a}sser variables reformulation of the equations. For both schemes, we established boundedness and convergence. Several numerical experiments are presented to show the effectiveness of the schemes.
\end{abstract}


\end{frontmatter}

\section{Introduction}\label{sec:Intro}
 Magnetohydrodynamics (MHD) describes the interaction between electrically conducting fluids and the electromagnetic fields. In an MHD system, the coupling between the flow and the electromagnetic field occurs through two mechanisms. Firstly, the electric current and the magnetic field generate the Lorentz force that acts on the flow. Secondly, the flow of conducting fluid itself induces a magnetic field, thereby altering the applied field. The MHD system has many practical applications, such as industrial liquid metals \cite{gerbeau2006mathematical}, metallurgy \cite{Davidson_2001}, dynamo \cite{Nore_2016}, etc.
 
 Mathematically, the stationary MHD system can be written as follows:
\begin{align}
- \nu\Delta \V{u} 
+ \V{u} \cdot \nabla \V{u} - \kappa \V{B}\cdot \nabla \V{B} - \gamma \nabla \divergence \V{u} + \nabla p  &= \V{f}, \label{eq:MHD_U} \\
\nabla \cdot \V{u} &= 0, \label{eq:MHD_Umass} \\
- \mu\Delta \V{B} + \V{u} \cdot \nabla \V{B} - \V{B} \cdot \nabla \V{u} - \gamma \nabla \divergence \V{B} + \nabla \lambda  &= \V{g}, \label{eq:MHD_B} \\
\nabla \cdot \V{B} &= 0, \label{eq:MHD_Bmass}
\end{align}
where $\V{u}$ is the velocity of the fluid, $\V{B}$ is the magnetic field, $p:=\hat{p}+\kappa\frac{|\V{B}|^2}{2}$ is the modified pressure with $\hat{p}$ being the pressure, $\lambda $ is a variable acting as a Lagrange multiplier corresponding to the solenoidal constraint on the magnetic field, $\V{f}$ is the body forcing, $ \V{g}$ is the divergence-free forcing on the magnetic field $\V{B}$, $\kappa$ is a coupling number, $\nu$ is the kinematic viscosity, and $\mu$ is the magnetic diffusivity. 
For simplicity, we consider homogeneous Dirichlet boundary conditions for both $\V{u}$ and $\V{B}$
\begin{equation}\label{BC}
\V{u} = \V{0} \quad \text{and } \quad \V{B}=\V{0} \quad \text{on } \; \partial \Omega.
\end{equation} 
The $\gamma$ terms in \eqref{eq:MHD_U}-\eqref{eq:MHD_Bmass} are the grad-div stabilization terms, which vanish for the continuous solution. They are well-known for enhancing and improving the numerical approximations of incompressible flows in various contexts, cf. \cite{https://doi.org/10.1002/fld.3654,GEREDELI2023114920,JenJohLinReb2014,OlsReu2004}.

With appropriate changes, our analysis can also be extended for no-slip velocity conditions together with $\V{B} \cdot \V{n} = 0$  and  $ (\nabla \times \V{B}) \times \V{n} = 0$  (in this case, the Maxwell equation uses the curl-curl form of the dissipation term).

The steady MHD system has been the subject of many studies. Among the earliest ones, in \cite{Gunzburger1991523}, Gunzburger et. al. studied the well-posedness for the continuous and discrete problem with nonhomogeneous boundary conditions. Moreover, they analyzed three linearization methods: the fully explicit method, Newton's method, and the semi-implicit Oseen's method. In \cite{XU2022105372}, the authors considered two-grid Newton's method with nonconforming finite element spaces. \cite{GREIF20102840} studied a DG-based mixed method, where the linearization was performed via the Picard iterations. \cite{Mengying2021} developed an algebraic splitting method, in the spirit of \cite{rebholz2019efficient}, where the Schur complement matrix is SPD and independent of iteration level. Yang et. al. \cite{YANG2019347} extended the Arrow-Hurwicz method of \cite{chen17} to stationary MHD equations. Due to the lagging of some dissipative terms, approaches of \cite{YANG2019347,chen17}  usually entail complicated assumptions on the scheme parameters. On the other hand, an improved version of the Arrow-Hurwicz method was developed in \cite{TC23,Calcolo2026}. Finally, we mention the article \cite{DONG2014287}, where, similar to the work of \cite{Gunzburger1991523}, three iterative methods of Stokes-, Newton-, and Oseen-type were investigated.

The goal of this paper is to present two iterative schemes inspired by the Incremental Viscosity Splitting scheme for unsteady Navier-Stokes \cite{Yak2023} flows. The idea has already been successfully applied to steady Navier-Stokes equations \cite{SIVS_Takhirov2026}. In this approach, the equations are split with respect to the viscous and resistive terms, respectively. This results in a PDE system where the intermediate velocity and magnetic fields must be solved for simultaneously, and then the end-of-step variables are obtained by solving two linear Stokes systems. These Stokes systems result in SPD Schur complement matrices that are assembled and preconditioned only once. Unlike the Arrow-Hurwicz scheme \cite{TC23}, IVS schemes still require a solution of a mixed problem. However, a gain is that the end-of-step fields are (discretely) divergence-free.

This paper is organized as follows. Section \ref{sec:Notations} presents the notations. The next Section \ref{sec:Existence} presents some well-posedness results for the MHD system \eqref{eq:MHD_U}-\eqref{eq:MHD_Bmass}. Afterwards, in Section \ref{sec:Picard}, we recall the properties of the classical Picard iterative scheme. Our novel schemes are presented in Sections \ref{sec:SIVS} and \ref{sec:SIVS_Els}, while the section \ref{sec:Numerics} is dedicated to numerical experiments. The last Section then concludes the manuscript.
\section{Notations and preliminaries}\label{sec:Notations}
Throughout this work, vector fields and spaces are denoted using boldface notation.
\noindent Standard notations for Sobolev spaces and corresponding norms will be used throughout the paper, see e.g., \cite{Ada1975}. In particular, $(\cdot,\cdot)$ and $\|\cdot \|$ denote $L^2(\Omega)$ inner product and the corresponding norm, respectively. $\V{H}^{k}$, where $k$ is an integer greater than zero, will denote the space of vector-valued functions each of whose $n$ components belong to $H^k$, the Sobolev space of real-valued functions with square integrable derivatives of order up to $k$ equipped with the usual norm $\|\cdot\|_k$. The dual space of $\V{H}^{1}_0(\Omega)$ will be denoted by $\V{H}^{-1}$, and the duality pairing between these two spaces is denoted by $\langle \cdot,\cdot \rangle$. The norm in 
$\V{H}^{-1}$ is given by $\|\V{f} \|_{-1} = \langle \V{f}, (-\Delta )^{-1}\V{f} \rangle^{1/2}.$

The equivalent weak formulation of \eqref{eq:MHD_U}-\eqref{eq:MHD_Bmass} reads as follows: $\forall (\V{v},q,\V{S},\eta)\in(\V{X},Q)^2 $, find $(\V{u},p,\V{B},\lambda)\in (\V{X},Q)^2$ satisfying 
\begin{align}
a_1(\V{u}, \V{v})+c^*(\V{u},\V{u},\V{v})-\kappa c^*(\V{B},\V{B},\V{v})+b(p,\V{v}) & = \left<\V{f},\V{v} \right>, \label{eq:exact1}\\
b(q,\V{u}) & = 0, \label{eq:exact2} \\
a_2(\V{B}, \V{S})+c^*(\V{u},\V{B},\V{S})-c^*(\V{B},\V{u},\V{S})+b(\lambda,\V{S}) & = \left<\V{g},\V{S} \right>, \label{eq:exact3} \\
b(\eta,\V{B}) & = 0, \label{eq:exact4}
\end{align}
where $\V{X}:= \V{H}_0^1(\Omega), \; Q:=L_0^2(\Omega)$ and 
\begin{align*}
a_1(\V{u}, \V{v}) & = \nu(\nabla \V{u}, \nabla \V{v}) + \gamma (\nabla \cdot \V{u}, \nabla \cdot \V{v}), \\
b(p, \V{v}) & = -(p,\nabla\cdot \V{v}), \\
a_2(\V{B}, \V{S}) & = \mu(\nabla \V{B}, \nabla \V{S}) + \gamma (\nabla \cdot \V{B}, \nabla \cdot \V{S}), \\
c(\V{u},\V{v},\V{w}) &= ((\V{u}\cdot \nabla)\V{v},\V{w}), \\
c^* (\V{u},\V{v},\V{w}) &= c(\V{u},\V{v},\V{w})+\frac{1}{2}\left({(\nabla \cdot \V{u})}\V{v}, \V{w} \right).
\end{align*}
The following bound holds for all $\V{u},\V{v}, \V{w}\in \V{X}$, see for instance,
\cite{GirRav1986,Tem1979}:
\begin{eqnarray}
c(\V{u},\V{v},\V{w})\leq \mathcal{M}_0 \|\nabla \V{u}\|\|\nabla \V{v}\|\|\nabla \V{w}\| \text{ and } c^*(\V{u},\V{v},\V{w})\leq \mathcal{M} \|\nabla \V{u}\|\|\nabla \V{v}\| \|\nabla \V{w}\|, \label{eq:nnl}
\end{eqnarray}
for some $\Mbd_0, \Mbd =\mathcal{O}(1)$. 
The assumption on $\Omega$ is sufficient to ensure that the inf--sup (or Ladyzhenskaya--Babu\v{s}ka--Brezzi, LBB) condition holds (see~\cite{BofBreFor2013,BoyFab2013,GirRav1986}): $\exists \beta=\beta(\Omega) > 0$ such that
\begin{equation*}
\inf_{q \in Q} \;
\sup_{\V{v} \in \V{X}}
\frac{b(q,\V{v})}{\|q\|\, \|\nabla \V{v}\|}
\geq \beta.
\end{equation*} 
We also define the div-free subspace of $\V{X}$:
\begin{equation*}
    \V{V}: = \{ \V{v} \in \V{X}:\;  b(q,\V{v})=0 \; \forall q\in Q\} = \{ \V{v} \in \V{X}:\; \nabla \cdot  \V{v}=0 \; \text{in} \; \Omega\},
\end{equation*}
For operators 
\begin{align*}
    \Lumap:= -\nu \Delta - \gamma \nabla \divergence:\; \V{X}\rightarrow \V{H}^{-1} , \, 
    \Lbmap:= -\mu \Delta - \gamma \nabla \divergence:\; \V{X}\rightarrow \V{H}^{-1},
\end{align*}
associated with the bilinear forms $a_1(\cdot,\cdot)$ and $a_2(\cdot,\cdot)$, respectively, we define:
\begin{equation}
    \label{eq:Tnorm}
    \begin{aligned}
    \|\V{g}\|_{\Lumap}: = \sqrt{\langle \Lumap \V{g},\V{g} \rangle} &\text{ and } 
    \|\V{f}\|_{\Luinvmap}: = \sqrt{\langle \Luinvmap \V{f},\V{f} \rangle}, \\
    \|\V{g}\|_{\Lbmap}: = \sqrt{\langle \Lbmap \V{g},\V{g} \rangle} &\text{ and } 
    \|\V{f}\|_{\Lbinvmap}: = \sqrt{\langle \Lbinvmap \V{f},\V{f} \rangle}.
    \end{aligned}    
\end{equation}
%
%
The following norm equivalences can be easily verified:
\begin{lemma}\cite{SIVS_Takhirov2026}
The following inequalities hold:
\begin{equation}
\begin{aligned}
\forall \V{v} \in \V{X}, \, \sqrt{\nu}\|\nabla \V{v}\| \le \| \V{v} \|_\Lumap \le \sqrt{\nu+\gamma} \|\nabla \V{v}\|, \, 
\sqrt{\mu}\|\nabla \V{v}\| \le \| \V{v} \|_\Lbmap \le \sqrt{\mu+\gamma} \|\nabla \V{v}\|, \\
\text{ and } \forall \V{f} \in \V{H}^{-1}, \, 
 \frac{1}{\sqrt{\nu+\gamma}}\| \V{f} \|_{-1} \le \| \V{f} \|_\Luinvmap \le \frac{1}{\sqrt{\nu}} \| \V{f} \|_{-1}, \, 
  \frac{1}{\sqrt{\mu+\gamma}}\| \V{f} \|_{-1} \le \| \V{f} \|_\Lbinvmap \le \frac{1}{\sqrt{\mu}} \| \V{f} \|_{-1}.
\end{aligned}
    \label{eq:NormEquiv}
\end{equation}    
\end{lemma}
We also equip the product space $\V{X}^2$ with a product norm:
\begin{equation}
    \label{eq:ProdNorm}
    \| (\V{u},\V{B})\|_\VX: = \sqrt{\nu \| \nabla \V{u}\|^2 + \kappa \mu \| \nabla \V{B}\|^2}.
\end{equation}
Next, we state two preliminary lemmas on non-negative sequences that will be used in the sequel, taken from \cite{Calcolo2026}.
\begin{lemma}[Sequences converging to $0$]\label{lem:SeqConv0}
 Assume that $\{a_k\}_{k=1}^\infty, \{b_k\}_{k=1}^\infty, \{c_k\}_{k=1}^\infty$ are non-negative sequences of real numbers and $\exists \, \omega_i$, $\varepsilon_i$, $i=\overline{1,2}$, such that $0<\varepsilon_i \le \omega_i$ and 
 \begin{equation*}
  \omega_1 a_{k+1} + \omega_2 b_{k+1} + c_{k+1} \leq \left(\omega_1-\varepsilon_1 \right) a_k + \left(\omega_2 - \varepsilon_2 \right) b_{k} + c_{k}.
 \end{equation*}
Then $\exists \, C \ge 0$ such that 
\begin{align*}
 \lim\limits_{n \rightarrow \infty} \left(a_k,b_k,c_k\right) & = (0,0,C).
\end{align*}
\end{lemma}
\begin{lemma}[Contractivity of sequences converging to $0$]\label{ContractivitySequence}
Assume that $\{a_k\}_{k=1}^\infty, \{b_k\}_{k=1}^\infty, \{c_k\}_{k=1}^\infty$ are non-negative sequences of real numbers and $\exists \, \omega_i$, $i=\overline{1,3}$, $\varepsilon_i$, $i=\overline{1,2}$, such that $0<\varepsilon_i \le \omega_i$, $i=1,2$, 
 \begin{equation}
  \omega_1 a_{k+1} + \omega_2 b_{k+1} + \omega_3 c_{k+1} \leq \left(\omega_1-\varepsilon_1 \right) a_k + \left(\omega_2 - \varepsilon_2 \right) b_{k} + \omega_3 c_{k}
  \label{eq:Contract1}
 \end{equation}
 and 
 \begin{equation}
 c_{k} \le \tau_1 a_{k+1} + \tau_2 a_{k} + \tau_3 b_{k+1} \text{ for some positive } \tau_i, i=\overline{1,3}.
 \label{eq:Contract2}
 \end{equation}
Then there exists a sequence that is a linear combination of $a_{k}, b_{k}, c_{k}$ and is contracting towards $0$.
\end{lemma}
\section{Well-posedness result for the MHD system \ref{eq:MHD_U}-\ref{eq:MHD_Bmass}}\label{sec:Existence}
The weak formulation of \eqref{eq:exact1}-\eqref{eq:exact4} in kernel spaces can be written as follows: $\forall (\V{v},\V{S}) \in \V{V}^2$:
\begin{align}
a_1(\V{u},\V{v}) + c(\V{u},\V{u},\V{v})  
- \kappa \, c(\V{B},\V{B},\V{v})   
&= \langle \V{f},\V{v} \rangle, \label{FV-NS} \\
a_2(\V{B}, \V{S}) + c(\V{u},\V{B},\V{S}) - 
c(\V{B},\V{u},\V{S}) &= 
\langle \V{g},\V{S}\rangle. \label{FV-Max}
\end{align}
First, we prove an a priori bound:
\begin{lemma}\label{lem:Stability}
Any solution $(\V{u},\V{B})$ of \eqref{FV-NS}-\eqref{FV-Max} satisfies 
\begin{equation}\label{eq:Stability}
 \|(\V{u},\V{B})\|_\VX \le \,
\sqrt{\frac{1}{\nu}  \|\V{f}\|^2_{-1}
 +  \frac{\kappa}{\mu}  \|\V{g}\|^2_{-1}}:= \; \Nbd_0.
\end{equation}
\end{lemma}
\begin{proof}
Take $\V{v}=\V{u}$ in \eqref{FV-NS} and $\V{S}=\kappa \, \V{B}$ in \eqref{FV-Max}, yielding 
\begin{align*}
 \|(\V{u},\V{B})\|_\VX^2 &= -c(\V{u},\V{u},\V{u})  
+ \kappa \,c(\V{B},\V{B},\V{u}) -
\kappa \,c(\V{u},\V{B},\V{B}) + 
\kappa \,c(\V{B},\V{u},\V{B})
+ \langle \V{f},\V{u} \rangle
+ \kappa \langle \V{g} , \V{B} \rangle.
 \end{align*}
Thanks to the incompressibility, the first four terms on the right-hand side vanish. Then the Cauchy-Schwarz gives \eqref{eq:Stability}.
\end{proof}
\begin{theorem}\label{The-existence}
For any data $(\V{f},\V{g}) \in \V{H}^{-1}(\Omega) \times \V{H}^{-1}(\Omega)$, the system \eqref{FV-NS}-\eqref{FV-Max} admits a solution satisfying \eqref{eq:Stability}. In addition, if 
\begin{equation}\label{Assum-WellPoss}
\Lambda_0:= \frac{2\Mbd_0 \Nbd_0}{\min\{\nu,\mu\}^{3/2}} < 1,
\end{equation}
then this solution is unique.
\end{theorem}
\begin{proof}
Since $\mathbb{V} \times \mathbb{V}$ is a closed subspace of 
$\V{H}^1_0(\Omega) \times \V{H}^1_0(\Omega) $, it is a separable Hilbert space. Hence there exists an increasing sequence of finite-dimensional subspaces $\mathbb{V}_m \times \mathbb{V}_m$ of $\mathbb{V}\times \mathbb{V} $ 
such that
  $$
  \mathbb{V}\times \mathbb{V} = 
  \underset{m\geq 0}{\bigcup}\mathbb{V}_m \times \mathbb{V}_m.
  $$
Let us define the following mapping $\Phi_m$ from 
$\mathbb{V}_m \times \mathbb{V}_m$ 	 into itself by:

\begin{equation} \label{MappPhi}
\begin{aligned}
\langle \Phi_m( \V{u}, \V{B}), (\V{v},\V{S})\rangle &:=
a_1(\V{u},\V{v}) + \kappa a_2(\V{B},\V{S})
+ c(\V{u},\V{u},\V{v})
\\ & 
- \kappa \,c(\V{B},\V{B},\V{v})
+ \kappa \,c(\V{u},\V{B},\V{S})
- \kappa \,c(\V{B},\V{u},\V{S})
 \\ & 
 - \langle \V{f}, \V{v} \rangle 
- \kappa \,\langle \V{g}, \V{S} \rangle.
 \end{aligned}
\end{equation}
The mapping $\Phi_m$ is continuous on $\mathbb{V}_m \times \mathbb{V}_m$.
Moreover, taking $(\V{v},\V{S})=(\V{u},\V{B})$ in \eqref{MappPhi} and using the Cauchy-Schwartz and Young's inequalities repeatedly, we obtain
\begin{equation*}
\begin{aligned}
\langle \Phi_m( \V{u}, \V{B}), (\V{u}, \V{B})\rangle &=
\|(\V{u},\V{B})\|^2_\VX 
 - \langle \V{f}, \V{u} \rangle 
- \kappa \,\langle \V{g}, \V{B} \rangle
\\ 
& \ge \|(\V{u},\V{B})\|^2_\VX 
- \|\V{f}\|_{-1} \|\nabla \V{u}\| - \kappa \|\V{g}\|_{-1} \|\nabla \V{B}\|
\\ &\ge 
\frac{1}{2}\|(\V{u},\V{B})\|^2_\VX   
- \frac{1}{2} \left( \frac{\|\V{f}\|^2_{-1}}{\nu} + \frac{\kappa \|\V{g}\|^2_{-1}}{\mu} \right) 
\\ & =
\frac{1}{2}\left( \|(\V{u},\V{B})\|^2_\VX   
- \Nbd_0^2 \right).
%
 \end{aligned}
\end{equation*}
So the right-hand side is nonnegative on the sphere of radius $r=\Nbd_0$. Thanks to the Brouwer fixed-point theorem (see for instance \cite{GirRav1986}) there exist  $(\V{u}_m,\V{B}_m) \in \mathbb{V} \times \mathbb{V} $ satisfying
\begin{equation}\label{Am=0}
 \Phi_m(\V{u}_m,\V{B}_m) = 0\qquad \mbox{ and } 
 \qquad 
\|(\V{u},\V{B})\|_\VX  \le r.    
\end{equation}
The sequence $((\V{u}_m,\V{B}_m))_m$ is uniformly bounded in 
$\mathbb{V} \times \mathbb{V} $. Then, there exists a sub-sequence, still denoted, for simplicity by $((\V{u}_m,\V{B}_m))_m$  that is weakly convergent in $\mathbb{V} \times \mathbb{V} $ towards an element $(\V{u},\V{B})$. 
Since the space $\mathbb{V} \times \mathbb{V} $ is compactly embedded in $\V{L}^4(\Omega)\times \V{L}^4(\Omega)$, the convergence 
$((\V{u}_m,\V{B}_m))_m$  is strong in $\V{L}^4(\Omega)\times \V{L}^4(\Omega)$. 

It remains to be checked that $(\V{u},\V{B})$ is a solution to the problem \eqref{FV-NS}-\eqref{FV-Max}.  To do this, we go back to 
\eqref{Am=0} to write for all $(\V{v},\V{S})$ in 
$\mathbb{V}_m \times \mathbb{V}_m$:
\begin{equation*} 
\begin{aligned}
0 &=
a_1(\V{u}_m,\V{v})
+ \kappa a_2(\V{B}_m, \V{S})
+ c(\V{u}_m,\V{u}_m,\V{v})
\\ & 
- \kappa \, c(\V{B}_m,\V{B}_m,\V{v})
+\kappa \, c(\V{u}_m,\V{B}_m,\V{S})
- \kappa \,c(\V{B}_m,\V{u}_m,\V{S})
 \\ & 
 - \langle \V{f}, \V{v} \rangle 
- \kappa \, \langle \V{g}, \V{S} \rangle.
 \end{aligned}
\end{equation*}
Let us check only the non-linear terms such as: $\displaystyle c(\V{u}_m,\V{u}_m, \V{v})=\int_\Omega \V{u}_m \cdot \nabla \V{u}_m \cdot \V{v} \, d\V{x}$. Since $\|(\V{u},\V{B})\|_\VX \le r $, we have that
\begin{eqnarray*}
\left|
\int_\Omega \V{u}_m \cdot \nabla \V{u}_m \cdot \V{v} \, d\V{x} 
- \int_\Omega \V{u} \cdot \nabla \V{u} \cdot \V{v} \, d\V{x} 
\right|
&=& 
\left|
\int_\Omega(\V{u}_m-\V{u})\cdot\nabla\V{u}_m\cdot\V{v}\;d\V{x}
+\int_\Omega(\V{u}\cdot\nabla)(\V{u}_m-\V{u})\cdot\V{v}\;d\V{x}\;d\V{x}  
\right|
\\ &\le& \frac{r}{\nu} \|\V{u}_m-\V{u}\|_{\V{L}^4} \|\V{v}\|_{\V{L}^4}
+\|\V{u}\|_{\V{L}^4}|\V{v}|_1\|\V{u}_m-\V{u}\|_{\V{L}^4}.
\end{eqnarray*}
The strong convergence $\V{u}_m \longrightarrow \V{u}$
in $\V{L}^4(\Omega) $ ensures that the right-hand side goes to zero as $m \rightarrow\infty$.
Similarly, we get the convergence of 
\begin{equation*}
- \int_\Omega \V{B}_m \cdot \nabla \V{B}_m \cdot \V{v} \, d\V{x}   
+  \int_\Omega \V{u}_m \cdot \nabla \V{B}_m \cdot \V{S} \, d\V{x} 
- \int_\Omega \V{B}_m \cdot \nabla \V{u}_m \cdot \V{S} \, d\V{x}     
\end{equation*}
to 
\begin{equation*}
- \int_\Omega \V{B} \cdot \nabla \V{B} \cdot \V{v} \, d\V{x}  
+  \int_\Omega \V{u} \cdot \nabla \V{B} \cdot \V{S} \, d\V{x} 
- \int_\Omega \V{B} \cdot \nabla \V{u} \cdot \V{S} \, d\V{x}.
\end{equation*}
We conclude that $(\V{u},\V{B})$ is a solution of the problem \eqref{FV-NS}-\eqref{FV-Max}.

To prove uniqueness, assume that there exist two solutions $(\V{u}_i,\V{B}_i),\, i=1,2,$ of \eqref{FV-NS}-\eqref{FV-Max}. Taking the difference between the two equations with test function $(\V{v},\V{S})$ equal to $(\delta \V{u},\kappa \,\delta \V{B} ) := (\V{u}_1-\V{u}_2,\kappa \, (\V{B}_1 - \V{B}_2))$ gives
\begin{align*}
0 & = \|(\delta \V{u},\delta \V{B})\|^2_\VX 
+ \underbrace{\int_\Omega \left( \V{u}_1 \cdot \nabla \V{u}_1  - 
\V{u}_2 \cdot \nabla \V{u}_2 \right) \cdot \delta \V{u} \, d\V{x} 
}_{I_1}
+ 
\kappa \underbrace{ \,\int_\Omega \left( \V{u}_1 \cdot \nabla \V{B}_1 
- \V{u}_2 \cdot \nabla \V{B}_2 \right) 
\cdot   \delta \V{B} \, d\V{x}
}_{I_2}
\\
& \qquad 
- \kappa \underbrace{\int_\Omega \left( \V{B}_1\cdot \nabla \V{B}_1-
\V{B}_2\cdot \nabla \V{B}_2\right)  \cdot \delta \V{u} \, d\V{x} }_{I_3}   
- \kappa \underbrace{\int_\Omega \left( \V{B}_1 \cdot \nabla \V{u}_1 -\V{B}_2 \cdot \nabla \V{u}_2 \right)
\cdot \delta \V{B} \, d\V{x}}_{I_4}. 
\end{align*}
Thanks to the incompressibility, the first and the second integrals $I_1$ and $I_2$ become 
\begin{align*}
I_1 & = 
c(\delta\V{u},\V{u}_1,\delta \V{u})
+ c( \V{u}_2,\delta \V{u},\delta \V{u})
\; = c(\delta \V{u}, \V{u}_1,\delta \V{u}), \\
I_2 
& = c(\delta  \V{u},\V{B}_1, \delta \V{B})
+  c(\V{u}_2,\delta \V{B},\delta \V{B}) \; = 
 c(\delta \V{u},\V{B}_1,\delta \V{B}). 
\end{align*}
On the other hand
\begin{align*}
I_3 & =
c(\delta \V{B}, \V{B}_1, \delta \V{u} )
+ c(\V{B}_2,\delta \V{B},\delta \V{u}) \; = 
 c(\delta \V{B}, \V{B}_1,\delta \V{u})
- c(\V{B}_2,\delta \V{u},\delta \V{B}),
\\
I_4 &= c(\delta \V{B},\V{u}_1,\delta \V{B})
+ c(\V{B}_2,\delta \V{u},\delta \V{B}). 
\end{align*}
Hence 
$$
I_3 + I_4 = 
c(\delta \V{B}, \V{B}_1,\delta \V{u})
 + c(\delta \V{B},\V{u}_1,\delta \V{B}). 
$$
Then, using \eqref{eq:Stability} and Young's inequality, we can write the following bounds
\begin{align*}
\|(\delta \V{u},\delta \V{B})\|^2_\VX 
& \le 
\left| c(\delta \V{u},\V{u}_1,\delta \V{u}) \right|
+ \kappa \left|  c(\delta \V{u},\V{B}_1, \delta \V{B}) \right| 
+  \kappa\,\left| c(\delta \V{B},\V{B}_1,\delta \V{u}) \right|
 + \kappa\, \left| c(\delta \V{B},\V{u}_1,\delta \V{B})\right|
\\ & \le 
\Mbd_0 \left(
 \|\nabla  \V{u}_1 \| \left( \|\nabla \delta \V{u} \|^2 + \kappa\,  \|\nabla \delta \V{B} \|^2 \right)
+ 2 \kappa\, \|\nabla  \V{B}_1 \| \, \|\nabla \delta \V{u} \|\, \|\nabla \delta \V{B} \|  \right) 
\\ & \le 
\Mbd_0 \Nbd_0 \left(  \frac{\|\nabla \delta \V{u} \|^2 + \kappa\,  \|\nabla \delta \V{B} \|^2}{\sqrt{\nu}} + \frac{\|\nabla \delta \V{u} \|^2 + \kappa\,  \|\nabla \delta \V{B} \|^2}{\sqrt{\mu}} \right) 
\\ & \le 
\frac{2\Mbd_0 \Nbd_0}{\min\{\nu,\mu\}^{3/2}} \|(\delta \V{u},\delta \V{B})\|^2_\VX .
\end{align*}
Recalling the small data assumption \eqref{Assum-WellPoss}, we obtain the uniqueness of the solution. 
\end{proof}
\section{Picard iterative method}\label{sec:Picard}
\DY{Having established uniqueness under the small data assumption \eqref{Assum-WellPoss}, we now focus on the constructive aspect of the problem. To this end, we develop a Picard iterative method based on the interaction between the two equations and show that it converges to the unique solution of the initial problem.}
\begin{algorithm}\label{algorithm:Picard}

Let $\V{u}_0=\mathbf{0}, \V{B}_0=\mathbf{0}$, and for $k = 1, 2, \ldots $, compute until convergence:  

Find $(\V{u}_k,\V{B}_k,p_k,\lambda_k) \in \VX^2 \times Q^2$ solution of 
\begin{equation}\label{eq:PicardScheme}
\begin{aligned}
& a_1(\V{u}_{k} ,\V{v})
+ c(\V{u}_{k-1},\V{u}_{k},\V{v})
- \kappa \, c(\V{B}_{k-1},\V{B}_{k},\V{v}) + b(p_{k},\V{v})
= \langle \V{f}, \V{v} \rangle 
\qquad \forall \, \V{v} \in \VX, \\
& a_2(\V{B}_{k},\V{S}) + 
c(\V{u}_{k-1},\V{B}_{k},\V{S}) - 
c(\V{B}_{k-1},\V{u}_{k},\V{S})
+ b(\lambda_{k},\V{S})
=  \langle \V{g}, \V{S} \rangle
\qquad \forall \, \V{S} \in \VX, \\  
& b(q, \V{u}_{k}) =0 \qquad \text{and}  \qquad 
b(\eta, \V{B}_{k}) =0
\qquad \forall \, (q,\eta) \in Q^2.
\end{aligned}
\end{equation}
\end{algorithm}
\begin{theorem}\label{Thm-MonoCVG}
Under the same assumption of Theorem \ref{The-existence}, the solution of the Picard scheme \eqref{eq:PicardScheme} is uniformly bounded   
\begin{equation}\label{Est-UkBk}
\|(\V{u}_k,\V{B}_k)\|_\VX \le \Nbd_0,
\qquad \forall k \in  \mathbb{N}^*.
\end{equation}
Furthermore, if \eqref{Assum-WellPoss} holds, 
then the Scheme \eqref{eq:PicardScheme} is convergent.
\end{theorem}
\begin{proof}
Taking $(\V{v},\V{S})= (\V{u}_k,\kappa\, \V{B}_k)$ in \eqref{eq:PicardScheme} and using incompressibility conditions, yielding 
\begin{align*}
& \|(\V{u}_k,\V{B}_k)\|^2_\VX 
- \kappa\,c(\V{B}_{k-1},\V{B}_{k},\V{u}_k)
- \kappa\,c(\V{B}_{k-1},\V{u}_{k},\V{B}_k)
= \langle \V{f}, \V{u}_k \rangle
+ \kappa\,\langle \V{g}, \V{B}_k \rangle.
\end{align*}
Since $\displaystyle c(\V{B}_{k-1}, \V{u}_{k}, \V{B}_k ) = - c(\V{B}_{k-1},\V{B}_{k},\V{u}_k)$,
for all $k \in \mathbb{N}^*$, we immediately have the bound \eqref{Est-UkBk}.

We introduce the errors as 
\begin{equation}
    \label{eq:Picard0}
    \begin{aligned}
\V{e}_k & =\V{u}-\V{u}_k, \, \delta_k = p-p_k, \\
\V{D}_k & =\V{B}-\V{B}_k, \, r_k = \lambda-\lambda_k.
    \end{aligned}
\end{equation}

Then taking the difference between the equations \eqref{eq:PicardScheme} and the exact solution, we find for all $ (\V{v},\V{S},q,\eta) \in \V{X}^2   \times  Q^2$ that
\begin{equation}\label{Differnce-Mono-kETm}
\begin{aligned}
 a_1 (\V{e}_{k},\V{v})
+ a_2(\V{D}_{k},\V{S})
& + \underbrace{[c(\V{u},\V{u},\V{v})  
-c(\V{u}_{k-1},\V{u}_{k},\V{v})]
 }_{J_1}
- \underbrace{[c(\V{B},\V{B},\V{v})  
-c(\V{B}_{k-1},\V{B}_{k},\V{v})]
}_{J_2}
\\ & 
+ 
\underbrace{[c(\V{u},\V{B},\V{S})  
-c(\V{u}_{k-1},\V{B}_{k},\V{S})]
}_{J_3} 
- 
\underbrace{[c(\V{B},\V{u},\V{S})  
-c(\V{B}_{k-1},\V{u}_{k},\V{S})]
}_{J_4} 
\\ & 
+ b(\delta_{k},\V{v})
+ b(r_{k},\V{S}) = 0. 
\end{aligned}
\end{equation}
Taking $(\V{v},\V{S}) = (\V{e}_k,\kappa\,\V{D}_k)$ in \eqref{Differnce-Mono-kETm}, and using the incompressibility, we get  
\begin{align*}
J_1 &= c(\V{u},\V{u},\V{e}_k)  
-c(\V{u}_{k-1},\V{u}_{k},\V{e}_k)
 = c(\V{e}_{k-1},\V{u},\V{e}_k) 
 + c(\V{u}_{k-1},\V{e}_k,\V{e}_k)
 \\ & =  
  c(\V{e}_{k-1},\V{u},\V{e}_k) ,
\\
J_2 &= 
c(\V{B},\V{B},\V{e}_k)  
-c(\V{B}_{k-1},\V{B}_{k},\V{e}_k)
 = c(\V{D}_{k-1},\V{B},\V{e}_k)  
 + c(\V{B}, \V{D}_k,\V{e}_k), 
 \\
J_3 &= 
c(\V{u},\V{B},\V{D}_k)  
-c(\V{u}_{k-1},\V{B}_{k},\V{D}_k)
 = c(\V{e}_{k-1},\V{B},\V{D}_k)
 + c(\V{u},\V{D}_{k},\V{D}_k)
\\ & = 
c(\V{e}_{k-1},\V{B},\V{D}_k),
 \\
J_4 &=
c(\V{B},\V{u},\V{D}_k)  
-c(\V{B}_{k-1},\V{u}_{k},\V{D}_k)
 = c(\V{D}_{k-1},\V{u},\V{D}_k)
+ c(\V{B},\V{e}_{k},\V{D}_k)
\\ & = 
c(\V{D}_{k-1},\V{u},\V{D}_k)
-c(\V{B},\V{D}_k,\V{e}_k),
\end{align*}
where the last terms in $J_2$ and $J_4$ cancel each other out. Inserting $J_i, i=1,\cdots, 4$ in \eqref{Differnce-Mono-kETm}, we obtain
\begin{equation*}
\begin{aligned}
\|(\V{e}_k,\V{D}_k)\|^2_\VX 
& = -c(\V{e}_{k-1},\V{u}, \V{e}_k)
+ \kappa\, c(\V{D}_{k-1},\V{B}, \V{e}_k)
\\ & 
\quad 
- \kappa\, c(\V{e}_{k-1},\V{B}, \V{D}_k)
+ \kappa\,c(\V{D}_{k-1},\V{u},\V{D}_k).
\end{aligned}
\end{equation*}
Thanks to Cauchy-Schwartz and Young's inequalities, combining with the estimate \eqref{Est-UkBk}, we find 
\begin{align*}
\|(\V{e}_k,\V{D}_k)\|^2_\VX
 & \le 
  \Mbd_0 \Nbd_0 \left(\frac{\|\nabla \V{e}_{k-1}\|
\|\nabla \V{e}_{k}\|}{\sqrt{\nu}}
 + 
  \frac{\sqrt{\kappa}\|\nabla \V{D}_{k-1}\| \, \|\nabla \V{e}_{k}\|}{\sqrt{\mu}} \right)
\\ & \quad +
   \Mbd_0 \Nbd_0 \left( \frac{\sqrt{\kappa}\|\nabla \V{D}_{k}\| \, \|\nabla \V{e}_{k-1}\|}{\sqrt{\mu}}
+  \frac{\kappa \|\nabla \V{D}_{k-1}\| \, \|\nabla \V{D}_{k}\|}{\sqrt{\nu}} \right)
\\ & \quad \le 
\frac{2\Mbd_0 \Nbd_0}{\min\{\nu,\mu\}^{3/2}}
\sqrt{\nu \|\nabla \V{e}_{k-1} \|^2 + \kappa \mu \|\nabla \V{D}_{k-1} \|^2 }\,
\sqrt{\nu \|\nabla \V{e}_{k} \|^2 + \kappa \mu \|\nabla \V{D}_{k} \|^2 } 
\\ &  \qquad = 
 \Lambda_0\, \|(\V{e}_{k-1},\V{D}_{k-1})\|_\VX \, \|(\V{e}_k,\V{D}_k)\|_\VX.
\end{align*}
Then we get that
\begin{align*}
\|(\V{e}_k,\V{D}_k)\|_\VX \le \Lambda_0\, \|(\V{e}_{k-1},\V{D}_{k-1})\|_\VX \le \Lambda_0^k \, \|(\V{u},\V{B})\|_\VX.
\end{align*}
Therefore, if $\Lambda_0<1 $, then the sequence $(\V{u}_k,\V{B}_k)$ converges to $(\V{u},\V{B}) $ as $k \rightarrow \infty$ in $\V{V} \times \V{V}$.

To show the convergence of the pressure approximations $p_k$, consider the error in the momentum equation $ \V{v} \in \V{X}$:
\begin{equation}\label{ErrPicardPressure}
\begin{aligned}
 a_1 (\V{e}_{k},\V{v})
& + [c(\V{e}_{k-1},\V{u},\V{v})  
+c(\V{u}_{k-1},\V{e}_{k},\V{v})]
- \kappa[c(\V{D}_{k-1},\V{B},\V{v})  
+c(\V{B},\V{D}_{k},\V{v})]
+ b(\delta_{k},\V{v})
 = 0. 
\end{aligned}
\end{equation}
Applying the inf-sup condition in \eqref{ErrPicardPressure} gives
\begin{equation}
    \label{ErrPicardPressure1}
     \beta \| \delta_k \| \le \nu\|\nabla \V{e}_{k}\| + \gamma \|\divergence \V{e}_{k}\| + \Mbd_0 \Nbd_0  \left(\frac{\|\nabla \V{e}_{k-1}\| +
\|\nabla \V{e}_{k}\|}{\sqrt{\nu}}
 + \frac{\sqrt{\kappa}\|\nabla \V{D}_{k-1}\| \ + \|\nabla \V{D}_{k}\|}{\sqrt{\mu}} \right),
\end{equation}
which also implies the convergence of the pressure sequence. The convergence  $\lambda_k \rightarrow \lambda$ can be shown in a similar way, and is omitted for brevity.
\end{proof}
\subsubsection{The linear system for solving Algorithm \ref{algorithm:Picard}}\label{subsec:PicardSystem}
To discuss the linear system corresponding to  \eqref{eq:PicardScheme}, assume $\kappa=1$, and let $(\overrightarrow{X}_{\V{u},k}, \overrightarrow{X}_{\V{B},k}) \in \mathbb{R}^n \times \mathbb{R}^n$, and $(\overrightarrow{X}_{p,k}, \overrightarrow{X}_{\lambda,k}) \in \mathbb{R}^m \times \mathbb{R}^m$ be the coefficient vectors corresponding to the spatial discretizations of $\V{u}_k, \V{B}_k$, and $p_k, \lambda_k$. Moreover, for $j=1,2$, let us define the corresponding finite element matrices: 
\begin{equation}
\label{eq:FEM_Matrices}
    \begin{aligned}
     \widetilde{A}_j &\in \mathbb{R}^{n \times n}, \text{ matrix arising from } a_j(\cdot,\cdot):\VX_h \times \VX_h, \\
     C_{1,k-1} &\in \mathbb{R}^{n \times n}, \text{ matrix arising from mixed terms of } c^*(\V{u}_{k-1},\cdot,\cdot):\VX_h \times \VX_h, \\
     C_{2,k-1} &\in \mathbb{R}^{n \times n}, \text{ matrix arising from mixed terms of } c^*(\V{B}_{k-1},\cdot,\cdot):\VX_h \times \VX_h, \\
     A_{j,k-1} & = \widetilde{A}_j + C_{j,k-1}, \\
     \widehat{B} &\in \mathbb{R}^{n \times m}, \text{ matrix arising from } b(\cdot,\cdot): \VX_h \times Q_h,
    \end{aligned}
\end{equation}
The linear system of the Picard method \eqref{eq:PicardScheme} is then
\begin{equation}\label{eq:PicardSystem1}
\begin{bmatrix} 
A_{1,k-1} & C_{2,k-1} & \widehat{B}^T & 0 \\ 
C_{1,k-1} & A_{2,k-1} &  0  & \widehat{B}^T \\
\widehat{B}         &  0        &  0  & 0 \\
0                   &  \widehat{B}  &  0  & 0
\end{bmatrix}
\begin{bmatrix} \overrightarrow{X}_{\V{u},k}  \\ 
\overrightarrow{X}_{\V{B},k}  \\
\delta \overrightarrow{X}_{p,k} \\
\delta \overrightarrow{X}_{\lambda,k}
\end{bmatrix}
= \begin{bmatrix} 
\overrightarrow{F}_1 - \widehat{B}^T \overrightarrow{X}_{p,k-1} \\ \overrightarrow{F}_2 - \widehat{B}^T \overrightarrow{X}_{\lambda,k-1} \\ \overrightarrow{0} \\ \overrightarrow{0}
\end{bmatrix}.
\end{equation}
The system \eqref{eq:PicardSystem1} is indefinite of size $(2n+2m) \times (2n+2m)$, which is not always easy to solve using iterative methods. By defining 
\begin{equation}
    \label{eq:PicardSystem2}
    \begin{aligned}
      A_{k-1} &:= \begin{bmatrix} 
        A_{1,k-1} & C_{2,k-1} \\ 
        C_{1,k-1} &  A_{2,k-1} 
        \end{bmatrix}, \, B:= \begin{bmatrix} 
        \widehat{B} & 0 \\ 
        0   &  \widehat{B}
        \end{bmatrix} \\
        \overrightarrow{X}_{1,k} &:= \begin{bmatrix}
            \overrightarrow{X}_{\V{u},k} \\
            \overrightarrow{X}_{\V{B},k}
        \end{bmatrix}, \, 
        \overrightarrow{X}_{2,k} := \begin{bmatrix}
            \overrightarrow{X}_{p,k} \\
            \overrightarrow{X}_{\lambda,k}
        \end{bmatrix},
    \end{aligned}
\end{equation}
the \eqref{eq:PicardSystem1} can be written in a block $LU$ form as
\begin{equation}\label{eq:PicardSystem3}
\begin{aligned}  
\begin{bmatrix} 
A_{k-1} & B^T \\ 
B &  0 
\end{bmatrix}
\begin{bmatrix} \overrightarrow{X}_{1,k}  \\ 
\delta \overrightarrow{X}_{2,k}
\end{bmatrix}
& = \begin{bmatrix} \overrightarrow{F} - B^T \overrightarrow{X}_{2,k} \\ \overrightarrow{0} \end{bmatrix} \implies \\
\begin{bmatrix} 
A_{k-1} & 0 \\ 
B &  -B A_{k-1}^{-1} B^T 
\end{bmatrix}
\begin{bmatrix} 
I & A_{k-1}^{-1} B^T \\
0 & I 
\end{bmatrix} 
\begin{bmatrix} 
\overrightarrow{X}_{1,k}  \\ 
\delta \overrightarrow{X}_{2,k}
\end{bmatrix}
& = 
\begin{bmatrix} \overrightarrow{F} - B^T \overrightarrow{X}_{2,k} \\ \overrightarrow{0} 
\end{bmatrix}.
\end{aligned}
\end{equation}
Theoretically, the latter system is easier to solve than the former, as smaller size linear systems with coefficient matrices, $A_{k-1} \in \mathbb{R}^{2n \times 2n}$ and $S_{k-1}:=-B A_{k-1}^{-1} B^T  \in \mathbb{R}^{2m \times 2m} $, needs be solved. However, solving the system with $S_{k-1}$ is known to be hard, and owing to its dependence on $k$, is computationally intensive, especially for 3D problems.
\section{Steady IVS scheme}\label{sec:SIVS}
\DY{In this section, we introduce an Incremental Viscosity Splitting (IVS) scheme and study its convergence properties. This approach provides an alternative numerical framework for approximating the solution.}
\begin{algorithm}\label{algorithm:SIVS}
Let $(\V{u}_0,p_0,\V{B}_0,\lambda_0) = (\V{0},0,\V{0},0)$ and for $k = 1, 2, \ldots $, compute the until convergence: 
\begin{enumerate}
\item[]{\bf Step 1:} Given $(\V{u}_{k-1}, \V{B}_{k-1}) \in \V{V}^2 $ and $(p_{k-1}, \lambda_{k-1}) \in Q^2$, 
find $(\Vt{u}_{k},\Vt{B}_{k}) \in \VX^2$ solution of 
\begin{equation}
\label{eq:SIVS1}
\begin{cases}
a_1(\Vt{u}_{k}, \V{v})
+ c^*(\V{u}_{k-1}, \Vt{u}_{k}, \V{v})
- \kappa c^*(\V{B}_{k-1}, \Vt{B}_{k}, \V{v})
+ b(p_{k-1},\V{v}) &= \langle \V{f}, \V{v} \rangle \quad \forall \, \V{v} \in \VX.\\
a_2(\V{\tilde{B}}_{k}, \V{S})
+ c^*(\V{u}_{k-1}, \Vt{B}_{k},\V{S})
- c^*(\V{B}_{k-1}, \Vt{u}_{k},\V{S})
+ b(\lambda_{k-1},\V{S}) &= \langle \V{g}, \V{S} \rangle, \quad \forall \, \V{S} \in \VX.
\end{cases}
\end{equation}
\item[]{\bf Step 2:} Find $(\V{u}_{k}, p_{k}) \in \VX  \times Q$ solution of 
\begin{equation}
\label{eq:SIVS2}
\begin{aligned}
a_1(\V{u}_{k} - \Vt{u}_{k}, \V{v}) + 
  b(p_{k}-p_{k-1}, \V{v}) & = 0
, \qquad \forall \, \V{v} \in \VX
 \\
b(q,\V{u}_{k}) & = 0, \qquad \forall \, q \in Q.
\end{aligned}
\end{equation}
\item[]{\bf Step 3:} Find $(\V{B}_{k}, \lambda_{k})  \in \VX  \times Q$ solution of 
\begin{equation}
\label{eq:SIVS3}
\begin{aligned}
a_2( \V{B}_{k} -  \V{\tilde{B}}_{k}, \V{S})
+ b(\lambda_{k}-\lambda_{k-1}, \V{S}) & = 0,
\quad \forall \, \V{S} \in \VX 
 \\
b(\eta, \V{B}_{k}) & = 0, \quad \forall \, \eta \in Q.
\end{aligned}
\end{equation}
\end{enumerate}
\end{algorithm}
We have the following uniform boundedness and convergence results for our Algorithm \ref{algorithm:SIVS}.
\begin{theorem}[Uniform boundedness]\label{boundedness}
If 
\begin{equation}
	\Lambda_1 := \frac{2\Mbd \Nbd_0}{\min\{\nu,\mu\}^{3/2}} < \frac{1}{\sqrt{2}} \label{eq:smalldata},    
\end{equation}
then $\| \nabla \Vt{u}_k\|, \| \nabla \V{u}_k\| , \| \nabla p_k\|_{-1}$, and  $\| \nabla \Vt{B}_k\|, \| \nabla \V{B}_k\| , \| \nabla \lambda_k\|_{-1}$ are uniformly bounded, and as $k \rightarrow \infty$ there holds 
\begin{equation}
    \label{eq:ErrConv_Els}
    \begin{aligned}
\V{u}_k \xrightarrow{\V{V}} \V{u}, \, \Vt{u}_k \xrightarrow{\V{X}} \V{u}, \text{ and } \, \nabla p_k \xrightarrow{\Luinvmap} \nabla p, \\
\V{B}_k \xrightarrow{\V{V}} \V{B}, \, \Vt{B}_k \xrightarrow{\V{X}} \V{B}, \text{ and } \, \nabla \lambda_k \xrightarrow{\Lbinvmap} \nabla \lambda,
    \end{aligned}
\end{equation}
where $(\V{u},p,\V{B},\lambda)$ is the unique solution of  \eqref{eq:MHD_U}--\eqref{eq:MHD_Bmass}.
\end{theorem} 
%
%
\begin{proof}
First, let us note that the uniqueness of the solution of \eqref{eq:MHD_U}--\eqref{eq:MHD_Bmass} follows from condition~\eqref{eq:smalldata}.
We define the errors as 
\begin{equation}
    \label{eq:err0}
    \begin{aligned}
\V{e}_k & =\V{u}-\V{u}_k, \, \Vt{e}_k=\V{u}-\Vt{u}_k \\
\delta_k & =p-p_k, \\
\V{D}_k & =\V{B}-\V{B}_k, \, \Vt{D}_k=\V{B}-\Vt{B}_k \\
r_k & = \lambda-\lambda_k.
    \end{aligned}
\end{equation}
Next we subtract \eqref{eq:SIVS1}-\eqref{eq:SIVS3} from \eqref{eq:MHD_U}-\eqref{eq:MHD_Bmass} to obtain the error equations $\forall \, \V{v} \in \V{X}$:
\begin{equation}\label{eq:err1}
\begin{aligned}  
a_1(\Vt{e}_{k},\V{v})
+ c^*(\V{e}_{k-1}, \V{u}, \V{v}) + c^*(\V{u}_{k-1}, \Vt{e}_k, \V{v})
- \kappa c^*(\V{D}_{k-1}, \V{B}, \V{v}) - \kappa c^*(\V{B}_{k-1}, \Vt{D}_k, \V{v})
+ b(\delta_{k-1},\V{v}) = 0, \\
a_2(\Vt{D}_{k},\V{S})
+ c^*(\V{e}_{k-1}, \V{B}, \V{S}) + c^*(\V{u}_{k-1}, \Vt{D}_k, \V{S})
- c^*(\V{D}_{k-1}, \V{u}, \V{S}) - c^*(\V{B}_{k-1}, \Vt{e}_k, \V{S})
+ b(r_{k-1},\V{S}) = 0,
\end{aligned}
\end{equation}
and
\begin{equation}\label{eq:err2}
\begin{aligned}
a_1(\V{e}_{k} - \Vt{e}_{k}, \V{v}) + b(\delta_k-\delta_{k-1},\V{v}) & = 0, \qquad \forall \, \V{v} \in \V{X}, \\
b(q, \V{e}_{k}) & = 0, \qquad \forall \, q \in Q,
\end{aligned}
\end{equation}
\begin{equation}\label{eq:err2.5}
\begin{aligned}
a_2(\V{D}_{k} - \Vt{D}_{k}, \V{S}) + b(r_k-r_{k-1},\V{S}) & = 0, \qquad \forall \, \V{S} \in \V{X}, \\
b(q, \V{D}_{k}) & = 0, \qquad \forall \, q \in Q.
\end{aligned}
\end{equation}
We first rewrite $b(\cdot, \cdot)$ terms. To this end, we test \eqref{eq:err2} with $\V{v} = \V{e}_k \in \V{V}$ to obtain that
\begin{equation}\label{eq:err3}
\nu \left( \|\nabla \V{e}_{k} \|^2 - \|\nabla \Vt{e}_{k} \|^2 + \|\nabla (\V{e}_{k}  - \Vt{e}_{k}) \|^2 \right) + 
\gamma \left( \|\divergence \V{e}_{k} \|^2 - \|\divergence \Vt{e}_{k} \|^2 + \|\divergence (\V{e}_{k}  - \Vt{e}_{k}) \|^2 \right) = 0.
\end{equation}
Moreover, the first equation of \eqref{eq:err2} implies that $\Vt{e}_{k}-\V{e}_{k} = \Luinvmap \nabla (\delta_{k}  - \delta_{k-1})$  in $\VX$, and that
\begin{equation}\label{eq:err4}
\|\nabla (\delta_{k}  - \delta_{k-1}) \|^2_{\Luinvmap} = \|\V{e}_{k}  - \Vt{e}_{k} \|^2_{\Lumap}
= \nu \|\nabla(\V{e}_{k}  - \Vt{e}_{k}) \|^2 + 
\gamma \|\divergence(\V{e}_{k}  - \Vt{e}_{k}) \|^2.
\end{equation}
Thus, 
\begin{equation}
    \label{eq:err5}
    \begin{aligned}        
b(\delta_{k-1}, \Vt{e}_k - \V{e}_k) & = \left\langle \nabla \delta_{k-1},  \Vt{e}_k - \V{e}_k \right \rangle \\ 
& = \left\langle \nabla \delta_{k-1}, \Luinvmap \nabla (\delta_{k}  - \delta_{k-1}) \right \rangle \\
& = \frac{1}{2} \left[\|\nabla \delta_{k} \|^2_{\Luinvmap} - \|\nabla \delta_{k-1} \|^2_{\Luinvmap} - \|\nabla (\delta_{k}  - \delta_{k-1}) \|^2_{\Luinvmap} \right] \\
& = \frac{1}{2} \left[\|\nabla \delta_{k} \|^2_{\Luinvmap} - \|\nabla \delta_{k-1} \|^2_{\Luinvmap} \right]  - \frac{\nu}{2}  \|\nabla (\V{e}_{k}  - \Vt{e}_{k}) \|^2
-\frac{\gamma}{2} \|\divergence(\V{e}_{k}  - \Vt{e}_{k}) \|^2 \\
& = \frac{1}{2} \left[\|\nabla \delta_{k} \|^2_{\Luinvmap} - \|\nabla \delta_{k-1} \|^2_{\Luinvmap} \right] +
\frac{\nu}{2} \left( \|\nabla \V{e}_{k} \|^2 - \|\nabla \Vt{e}_{k} \|^2 \right) + 
\frac{\gamma}{2} \left( \|\divergence \V{e}_{k} \|^2 - \|\divergence \Vt{e}_{k} \|^2\right),
    \end{aligned}
\end{equation}
where the last equality is deduced thanks to \eqref{eq:err3}. Similarly, 
\begin{equation}
    \label{eq:err5.5}
b(r_{k-1}, \Vt{D}_k - \V{D}_k) 
= \frac{1}{2} \left[\|\nabla r_{k} \|^2_{\Lbinvmap} - \|\nabla r_{k-1} \|^2_{\Lbinvmap} \right] + \frac{\mu}{2}\left( \|\nabla \V{D}_{k} \|^2 - \|\nabla \Vt{D}_{k} \|^2 \right) + 
\frac{\gamma}{2} \left( \|\divergence \V{D}_{k} \|^2 - \|\divergence \Vt{D}_{k} \|^2\right).
\end{equation}
By picking $(\V{v},\V{S})=(\Vt{e}_k,\Vt{D}_k)$ in \eqref{eq:err1} and using the skew-symmetry of the trilinear form $c^*(\cdot,\cdot,\cdot)$, we have
\begin{equation}\label{eq:err6}
\begin{aligned}  
a_1(\Vt{e}_{k},\Vt{e}_{k})
+ c^*(\V{e}_{k-1}, \V{u}, \Vt{e}_{k})
- \kappa c^*(\V{D}_{k-1}, \V{B}, \Vt{e}_{k}) - \kappa c^*(\V{B}_{k-1}, \Vt{D}_k, \Vt{e}_{k})
+ b(\delta_{k-1},\Vt{e}_{k}-\V{e}_k) = 0, \\
 a_2(\Vt{D}_{k},\Vt{D}_{k})
+  c^*(\V{e}_{k-1}, \V{B}, \Vt{D}_{k}) - c^*(\V{D}_{k-1}, \V{u}, \Vt{D}_{k}) - c^*(\V{B}_{k-1}, \Vt{e}_k, \Vt{D}_{k})
+  b(r_{k-1},\Vt{D}_{k}-\V{D}_{k}) = 0.
\end{aligned}
\end{equation}
We then multiply the second equation above by $\kappa$ and it to the first equation above, cancel the $\kappa c^*(\V{B}_{k-1}, \Vt{D}_k, \Vt{e}_{k})$ terms, and use the standard bounds to get
\begin{equation}\label{eq:err6.5}
\begin{aligned}  
\| (\Vt{e}_{k},\Vt{D}_{k}) \|^2_\V{X}
& + \gamma \left( \| \divergence \Vt{e}_{k} \|^2 + \| \divergence \Vt{D}_{k} \|^2 \right)
+ b(\delta_{k-1},\Vt{e}_{k}-\V{e}_k)
+ \kappa b(r_{k-1},\Vt{D}_{k}-\V{D}_{k}) 
\\
& = - c^*(\V{e}_{k-1}, \V{u}, \Vt{e}_{k})
+ \kappa c^*(\V{D}_{k-1}, \V{B}, \Vt{e}_{k}) 
- \kappa c^*(\V{e}_{k-1}, \V{B}, \Vt{D}_{k}) + \kappa c^*(\V{D}_{k-1}, \V{u}, \Vt{D}_{k}) 
\\ & \le
\Mbd \|\nabla \V{u} \| \left[\|\nabla \V{e}_{k-1} \|\, \|\nabla \Vt{e}_{k} \| + \kappa\|\nabla \V{D}_{k-1} \|\, \|\nabla \Vt{D}_{k} \| \right] 
\\ & + 
\Mbd \kappa \|\nabla \V{B} \| \left[\|\nabla \V{e}_{k-1} \|\, \|\nabla \Vt{D}_{k} \| + \|\nabla \V{D}_{k-1} \|\, \|\nabla \Vt{e}_{k} \| \right] 
\\ & \le 
\Mbd \Nbd_0\left[ \frac{\|\nabla \V{e}_{k-1} \|\, \|\nabla \Vt{e}_{k} \| + \kappa\|\nabla \V{D}_{k-1} \|\, \|\nabla \Vt{D}_{k} \|}{\sqrt{\nu}}
+ \sqrt{\kappa}\frac{\|\nabla \V{e}_{k-1} \|\, \|\nabla \Vt{D}_{k} \| + \|\nabla \V{D}_{k-1} \|\, \|\nabla \Vt{e}_{k} \|}{\sqrt{\mu}}\right]
\\ & \le
\frac{2\Mbd \Nbd_0}{\min\{\nu,\mu\}^{3/2}}
\sqrt{\nu \|\nabla \V{e}_{k-1} \|^2 + \kappa \mu \|\nabla \V{D}_{k-1} \|^2 }\,
\sqrt{\nu \|\nabla \Vt{e}_{k} \|^2 + \kappa \mu \|\nabla \Vt{D}_{k} \|^2 } 
\\ & \le
\frac{\| (\Vt{e}_{k},\Vt{D}_{k}) \|^2_\V{X}}{4}
+ \Lambda_1^2 \| (\V{e}_{k-1},\V{D}_{k-1}) \|^2_\V{X}.
\end{aligned}
\end{equation}
Combine the last identity with \eqref{eq:err5} and \eqref{eq:err5.5} to get
\begin{equation}\label{eq:err7}   
\begin{aligned}
\frac{1}{2} \| (\V{e}_{k},\V{D}_{k}) \|^2_\V{X} +  \frac{1}{4} \| (\Vt{e}_{k},\Vt{D}_{k}) \|^2_\V{X} 
& + \frac{\gamma}{2}\left( \| \nabla \cdot \Vt{e}_{k} \|^2 + \| \nabla \cdot \V{e}_{k} \|^2 + \kappa \| \nabla \cdot \Vt{D}_{k} \|^2 + \kappa \| \nabla \cdot \V{D}_{k} \|^2 \right) \\
& + \frac{ \|\nabla \delta_{k} \|^2_{\Luinvmap} + \kappa \|\nabla r_{k} \|^2_{\Lbinvmap}}{2} \\
& \le \frac{ \|\nabla \delta_{k-1} \|^2_{\Luinvmap} + \kappa \|\nabla r_{k-1} \|^2_{\Lbinvmap}}{2} +  \Lambda_1^2 \| (\V{e}_{k-1},\V{D}_{k-1}) \|^2_\V{X}.
\end{aligned}
\end{equation}
Now, thanks to the small data condition \eqref{eq:smalldata}, we can conclude that
\begin{equation}\label{eq:err8}   
\begin{aligned}
\frac{1}{2} \| (\V{e}_{k},\V{D}_{k}) \|^2_\V{X} +  \frac{1}{4} \| (\Vt{e}_{k},\Vt{D}_{k}) \|^2_\V{X} 
& + \frac{\gamma}{2}\left( \| \nabla \cdot \Vt{e}_{k} \|^2 + \| \nabla \cdot \V{e}_{k} \|^2 + \kappa \| \nabla \cdot \Vt{D}_{k} \|^2 + \kappa \| \nabla \cdot \V{D}_{k} \|^2 \right) \\
& + \frac{ \|\nabla \delta_{k} \|^2_{\Luinvmap} + \kappa \|\nabla r_{k} \|^2_{\Lbinvmap}}{2} \\
& \le \frac{ \|\nabla \delta_{0} \|^2_{\Luinvmap} + \kappa \|\nabla r_{0} \|^2_{\Lbinvmap}}{2} +  \Lambda_1^2 \| (\V{e}_{0},\V{D}_{0}) \|^2_\V{X} \\
& = \frac{ \|\nabla p_{0} \|^2_{\Luinvmap} + \kappa \|\nabla \lambda_{0} \|^2_{\Lbinvmap}}{2} +  \Lambda_1^2 \| (\V{u}_{0},\V{B}_{0}) \|^2_\V{X}.
\end{aligned}
\end{equation}
By the triangle inequality, we obtain the uniform boundedness of the solution:
\begin{equation}
\label{eq:err9}
\exists K_i>0, \, i=\overline{1,4}, \text{ with } \| (\V{u}_{k},\V{B}_{k}) \|_\V{X} \leq K_1, \, \| (\Vt{u}_{k},\Vt{B}_{k}) \|^2_\V{X} \leq K_2, \, 
\| \nabla p_k \|_{\Luinvmap} \le K_3 \text{ and } \| \nabla r_k \|_{\Lbinvmap} \le K_4.     
\end{equation}
Then, Lemma~\ref{lem:SeqConv0} implies that 
\begin{equation}
    \label{eq:err9.5}
    \begin{aligned}        
    \lim \limits_{k \rightarrow \infty} \| (\V{e}_{k},\V{D}_{k}) \|_\V{X} = \lim \limits_{k \rightarrow \infty} \| (\Vt{e}_{k},\Vt{D}_{k}) \|_\V{X} & = 0, \\ 
    \lim \limits_{k \rightarrow \infty} \left(\|\divergence \V{e}_{k} \|^2+\|\divergence \Vt{e}_{k} \|^2 + \kappa \|\divergence \V{D}_{k} \|^2 + \kappa\|\divergence \Vt{D}_{k} \|^2 \right) & = 0,
    \end{aligned}
\end{equation}
where we took
\begin{equation} \label{eq:err10}  
\begin{aligned}
a_k & = \| (\V{e}_{k},\V{D}_{k}) \|_\V{X}^2, \, b_k = \frac{1}{4}\| (\Vt{e}_{k},\Vt{D}_{k}) \|^2_\V{X}+ \frac{\gamma}{2} \left( \|\divergence \Vt{e}_{k} \|^2 + \|\divergence \V{e}_{k} \|^2 + \kappa \|\divergence \V{D}_{k} \|^2 + \kappa\|\divergence \Vt{D}_{k} \|^2\right), \\
c_k & = \frac{ \|\nabla \delta_{k} \|^2_{\Luinvmap} + \kappa \|\nabla r_{k} \|^2_{\Lbinvmap}}{2}, \\
\omega_1 & = \dfrac{1}{2},  \omega_2=1, \, \varepsilon_1 = \dfrac{1}{2} - \Lambda_1^2>0, \, \varepsilon_2 = 0.
\end{aligned}
\end{equation}
By applying the inf-sup condition in each equation of \eqref{eq:err1}, and the norm equivalence \eqref{eq:NormEquiv}, we obtain that
\begin{equation}
\label{eq:err11}
  \begin{aligned}
   \|\nabla \delta_k\|_{\Luinvmap} & \le C_1 \| (\Vt{e}_k,\Vt{D}_k) \|_{\V{X}} + 
   C_2 \| (\V{e}_{k-1},\V{D}_{k-1}) \|_{\V{X}}
   + \gamma \| \divergence \Vt{e}_k \|, 
   \\
   \|\nabla r_k\|_{\Lbinvmap} & \le C_3 \| (\Vt{e}_k,\Vt{D}_k) \|_{\V{X}} + 
   C_4 \| (\V{e}_{k-1},\V{D}_{k-1}) \|_{\V{X}}
   + \gamma \| \divergence \Vt{B}_k \,|   
    \end{aligned}
\end{equation}
for some positive constants $C_j$, $j=\overline{1,4}$.
Then \eqref{eq:err11} implies a bound equivalent to \eqref{eq:Contract2}. By invoking Lemma \ref{ContractivitySequence}, we can conclude geometric convergence. 

Ne{\'c}as' inequality
   \begin{equation}
        \label{eq:Necas}
      \exists C > 0 \;  \forall q \in Q: \| \nabla  q \|_{-1} \ge C \| q\|,
    \end{equation}
also shows the boundedness and convergence of the Lagrange multipliers in the $\| \cdot \|$ norm.    
\end{proof}
\subsubsection{The linear system for solving Algorithm \ref{algorithm:SIVS}}\label{subsec:SIVSSystem}
Keeping the notations of Subsection \ref{subsec:PicardSystem}, 
the system \eqref{eq:SIVS1} is equivalent to solving 
\begin{equation}
    \label{eq:SIVSSystem1}
    A_{k-1}\overrightarrow{X}_{1,k} =\overrightarrow{F} - B^T,
\end{equation}
while the equation \eqref{eq:SIVS2} is equivalent to
\begin{equation}\label{eq:SIVSSystem2}
\begin{aligned}  
\begin{bmatrix} 
\widetilde{A}_1 & \widehat{B}^T \\ 
\widehat{B} &  0 
\end{bmatrix}
\begin{bmatrix} \overrightarrow{X}_{\V{u},k}  \\ 
\delta \overrightarrow{X}_{p,k}
\end{bmatrix}
& = \begin{bmatrix} \widetilde{A}_1\overrightarrow{X}_{\Vt{u},k} - B^T \overrightarrow{X}_{p,k-1} \\ \overrightarrow{0} \end{bmatrix} \implies \\
\begin{bmatrix} 
\widetilde{A}_1 & 0 \\ 
\widehat{B} &  -\widehat{B} \widetilde{A}_1^{-1} \widehat{B}^T 
\end{bmatrix}
\begin{bmatrix} 
I & \widetilde{A}_1^{-1} \widehat{B}^T \\
0 & I 
\end{bmatrix} 
\begin{bmatrix} 
\overrightarrow{X}_{\V{u},k}  \\ 
\delta \overrightarrow{X}_{p,k}
\end{bmatrix}
& = 
\begin{bmatrix} \widetilde{A}_1\overrightarrow{X}_{\Vt{u},k} - B^T \overrightarrow{X}_{p,k-1} \\ \overrightarrow{0} 
\end{bmatrix}.
\end{aligned}
\end{equation}
In this case, the Schur matrix $S_1:=-\widehat{B} \widetilde{A}^{-1} \widehat{B}^T  \in \mathbb{R}^{2m \times 2m} $ is independent of $k$, and thus needs to be assembled and preconditioned only once. The linear system arising from \eqref{eq:SIVS3} is handled similarly.
\section{SIVS - Els{\"a}sser formulation}\label{sec:SIVS_Els}
Even though the Steady Incremental Viscosity Splitting Algorithm \ref{algorithm:SIVS} requires a solution of simpler linear systems compared to the classical Picard Algorithm \ref{algorithm:Picard}, the first step \eqref{eq:SIVS1} couples the intermediate velocity and magnetic fields $(\Vt{u}_{k},\Vt{B}_{k})$, which requires a solution of a large linear system. In this section, we will construct a novel IVS scheme based on the Els{\"a}sser reformulation of the MHD system \eqref{eq:MHD_U}-\eqref{eq:MHD_Bmass} that replaces \eqref{eq:SIVS1} with two decoupled linear systems of a smaller size.

To this end, we introduce the following Els{\"a}sser variables \cite{AGGUL202372}:
\begin{equation}
    \label{eq:ElsVariables}
    \begin{aligned}
    \V{w} &:= \V{u} + \sqrt{\kappa}\V{B}, \, \pplus: = p + \sqrt{\kappa} \lambda, \\
    \V{z} &:= \V{u}  - \sqrt{\kappa}\V{B}, \, \pminus: = p - \sqrt{\kappa} \lambda, \\
    \fplus &:=  \V{f} + \sqrt{\kappa} \V{g}\, , \fminus:=  \V{f} - \sqrt{\kappa}\V{g}, \\
    \nuplus &:=\frac{\nu+\mu}{2}, \, \numinus :=\frac{\nu-\mu}{2}.
    \end{aligned}
\end{equation}
To obtain a system for $(\V{w},\pplus,\V{z},\pminus)$, we first scale the equation \eqref{eq:MHD_B} by $\sqrt{\kappa}$, then add it to and subtract it from \eqref{eq:MHD_U}, resulting in the following system:
\begin{align}
 - \nuplus \Delta \V{w} + \V{z} \cdot \nabla \V{w} + \nabla \pplus 
 & = \fplus  + \numinus \Delta \V{z}
, \label{eq:Els1} \\
\nabla \cdot \V{w} &= 0, \label{eq:mass_Els1} 
\\
- \nuplus \Delta \V{z} + \V{w} \cdot \nabla \V{z} + \nabla \pminus  
& 
=  \fminus + \numinus \Delta \V{w}, \label{eq:Els2} \\
\nabla \cdot \V{z} &= 0. \label{eq:mass_Els2}
\end{align}
First, we prove an a priori bound for the solution of the system \eqref{eq:Els1}-\eqref{eq:mass_Els2}:
\begin{lemma}\label{lem:Stability_Els}
Any solution $(\V{w},\V{z})$ of \eqref{eq:Els1}-\eqref{eq:mass_Els2} satisfies 
\begin{equation}\label{eq:Stability_Els}
 \|(\nabla \V{w},\nabla \V{z})\|:=\sqrt{\| \nabla \V{w} \|^2 + \| \nabla \V{z} \|^2} \le \,
\frac{2\nuplus}{\nu\, \mu} 
\sqrt{\|\fplus\|^2_{-1}+\|\fminus\|^2_{-1}} := \; \Nbd_1.
\end{equation}
Additionally, if 
\begin{equation}
	\Lambda_2 := \frac{\Mbd \Nbd_1+|\numinus|}{\nuplus} < 1,
    \label{eq:smalldataUniq_Els}
\end{equation}
then the solution of \eqref{eq:Els1}-\eqref{eq:mass_Els2} is unique. 
\end{lemma}
\begin{proof}
Multiply \eqref{eq:Els1} by $\V{v}=\V{w}$ and integrate it over $\Omega$, yielding 
\begin{equation}
\label{eq:Els_Apriori1}
\begin{aligned}
\nuplus \|\nabla \V{w}\|^2 &= 
\langle \fplus,\V{w} \rangle
- \numinus (\nabla \V{z},\nabla \V{w})\\
& \le \frac{\nuplus}{\, \nu \, \mu}\|\fplus\|^2_{-1} + \frac{\nu \mu}{4\nuplus}\|\nabla \V{w}\|^2
+ \frac{\nuplus}{2}\|\nabla \V{w}\|^2
+ \frac{(\numinus)^2}{2\nuplus}\|\nabla \V{z}\|^2 \implies \\
\frac{2 (\numinus)^2+\mu\nu}{4 \nuplus}\|\nabla \V{w}\|^2 & \le \frac{\nuplus}{\, \nu \, \mu}\|\fplus\|^2_{-1}
+ \frac{(\numinus)^2}{2\nuplus}\|\nabla \V{z}\|^2.
 \end{aligned}
 \end{equation}
 Similarly, multiplying \eqref{eq:Els2} by $\V{v}=\V{z}$ and integrating gives
\begin{equation}
    \label{eq:Els_Apriori2}
    \frac{2 (\numinus)^2+\mu\nu}{4 \nuplus}\|\nabla \V{z}\|^2  \le \frac{\nuplus}{\, \nu \, \mu}\|\fminus\|^2_{-1}
+ \frac{(\numinus)^2}{2\nuplus}\|\nabla \V{w}\|^2.
\end{equation}
Combining \eqref{eq:Els_Apriori1} and \eqref{eq:Els_Apriori2} yields the bound \eqref{eq:Stability_Els}. 

To show the uniqueness of the solution, assume $(\V{w}_1,\V{z}_1)$ and $(\V{w}_2,\V{z}_2)$ solve \eqref{eq:Els1}-\eqref{eq:mass_Els2}. Then their differences $\delta \V{w}:=\V{w}_1-\V{w}_2$ and $\delta \V{z}:=\V{z}_1-\V{z}_2$ satisfy:
\begin{equation}\label{eq:Uniquness_Els1}
\begin{aligned}  
\nuplus (\| \nabla \delta \V{w}\|^2+\| \nabla \delta \V{z}\|^2)
& = - c^*(\delta \V{z}, \V{w}_1, \delta \V{w}) - c^*(\delta \V{w}, \V{z}_1, \delta \V{z}) - 2\numinus (\nabla \delta \V{z}, \nabla  \delta \V{w}) \\ 
& \le 2(\Nbd_1 \Mbd + |\numinus|) \| \nabla \delta \V{w}\|\, \| \nabla \delta \V{z}\| \\ 
& \le \frac{\Nbd_1 \Mbd + |\numinus|}{\nuplus}(\| \nabla \delta \V{w}\|^2+\| \nabla \delta \V{z}\|^2).
\end{aligned}
\end{equation}
Owing to the assumption \eqref{eq:smalldataUniq_Els}, we obtain the uniqueness of the solution.
\end{proof}
 \begin{algorithm}\label{algorithm:SIVS_Els}
Let $(\V{w}_0,\pplus_0,\V{z}_0,\pminus_0) = (\V{0},0,\V{0},0)$ and for $k = 1, 2, \ldots $, compute the following steps until convergence: 
\begin{enumerate}
\item[]{\bf Step 1:} Given $(\V{z}_{k-1}, \pplus_{k-1}) \in \V{V} \times Q$, find $\Vt{w}_{k} \in \VX$ solution of 
\begin{equation}
\label{eq:SIVS1_Els}
a_1(\Vt{w}_{k}, \V{v})
+ c^*(\V{z}_{k-1}, \Vt{w}_{k}, \V{v})
+ b(\pplus_{k-1},\V{v}) = \langle \fplus, \V{v} \rangle - \numinus(\nabla \V{z}_{k-1},\nabla \V{v}), \quad \forall \, \V{v} \in \VX.
\end{equation}
\item[]{\bf Step 2:} Find $(\V{w}_{k}, \pplus_{k}) \in \VX  \times Q$ solution of 
\begin{equation}
\label{eq:SIVS2_Els}
\begin{aligned}
a_1(\V{w}_{k} - \Vt{w}_{k}, \V{v}) + 
  b(\pplus_{k}-\pplus_{k-1}, \V{v}) & = 0
, \qquad \forall \, \V{v} \in \VX, \\
b(q,\V{w}_{k}) & = 0, \qquad \forall \, q \in Q.
\end{aligned}
\end{equation}
\item[]{\bf Step 3:} Given $(\V{w}_{k-1}, \pminus_{k-1}) \in \V{V} \times Q$, find $\Vt{z}_{k} \in \VX$ solution of 
\begin{equation}
\label{eq:SIVS3_Els}
a_1(\Vt{z}_{k}, \V{v})
+ c^*(\V{w}_{k-1}, \Vt{z}_{k}, \V{v})
+ b(\pminus_{k-1},\V{v}) = \langle \fminus, \V{v} \rangle - \numinus(\nabla \V{w}_{k-1},\nabla \V{v}), \quad \forall \, \V{v} \in \VX.
\end{equation}
\item[]{\bf Step 4:} Find $(\V{z}_{k}, \pminus_{k}) \in \VX  \times Q$ solution of 
\begin{equation}
\label{eq:SIVS4_Els}
\begin{aligned}
a_1(\V{z}_{k} - \Vt{z}_{k}, \V{v}) + 
  b(\pminus_{k}-\pminus_{k-1}, \V{v}) & = 0
, \qquad \forall \, \V{v} \in \VX, \\
b(q,\V{z}_{k}) & = 0, \qquad \forall \, q \in Q.
\end{aligned}
\end{equation}
\end{enumerate}
\end{algorithm}
Next, we establish the uniform boundedness and convergence results for our Algorithm \ref{algorithm:SIVS_Els}.
\begin{theorem}[Uniform boundedness]\label{boundedness}
If 
\begin{equation}
	\Lambda_3 := \frac{\Mbd \Nbd_1+|\numinus|}{\nuplus} < \frac{1}{\sqrt{2}},
    \label{eq:smalldata_Els}
\end{equation}
then $\| \nabla \Vt{w}_k\|, \| \nabla \V{w}_k\| , \| \nabla \pplus_k\|_{-1}$, and  $\| \nabla \Vt{z}_k\|, \| \nabla \V{z}_k\| , \| \nabla \pminus_k\|_{-1}$ are uniformly bounded, and as $k \rightarrow \infty$ there holds 
\begin{equation}
    \label{eq:ErrConv_Els}
    \begin{aligned}
\V{w}_k \xrightarrow{\V{V}} \V{w}, \, \Vt{w}_k \xrightarrow{\V{X}} \V{w}, \text{ and } \, \nabla \pplus_k \xrightarrow{\Luinvmap} \nabla \pplus, \\
\V{z}_k \xrightarrow{\V{V}} \V{z}, \, \Vt{z}_k \xrightarrow{\V{X}} \V{z}, \text{ and } \, \nabla \pminus_k \xrightarrow{\Lbinvmap} \nabla \pminus,
    \end{aligned}
\end{equation}
where $(\V{w},\pplus,\V{z},\pminus)$ is the unique solution of  \eqref{eq:Els1}--\eqref{eq:mass_Els2}.
\end{theorem} 
%
%
\begin{proof}
First, let us note that the uniqueness of the solution of \eqref{eq:Els1}--\eqref{eq:mass_Els2} follows from condition~\eqref{eq:smalldata_Els}.

Define the errors as 
\begin{equation}
    \label{eq:err0_Els}
    \begin{aligned}
\V{e}_k & =\V{w}-\V{w}_k, \, \Vt{e}_k=\V{w}-\Vt{w}_k \\
\delta_k & = \pplus-\pplus_k, \\
\V{D}_k & =\V{z}-\V{z}_k, \, \Vt{D}_k=\V{z}-\Vt{z}_k \\
r_k & = \pminus-\pminus_k.
    \end{aligned}
\end{equation}
Next we obtain the equations satisfied by $(\Vt{e}_k,\V{e}_k,\delta_k)$ and $(\Vt{D}_k,\V{D}_k,r_k)$ by subtracting \eqref{eq:SIVS1_Els}-\eqref{eq:SIVS4_Els} from \eqref{eq:Els1}-\eqref{eq:mass_Els2}:
\begin{equation}\label{eq:err1_Els}
a_1(\Vt{e}_{k},\V{v})
+ c^*(\V{D}_{k-1}, \V{w}, \V{v}) + c^*(\V{z}_{k-1}, \Vt{e}_k, \V{v})
+ b(\delta_{k-1},\V{v}) = -\numinus(\nabla \V{D}_{k-1},\nabla \V{v}) 
\, \forall \V{v} \in \VX,
\end{equation}
\begin{equation}\label{eq:err2_Els}
\begin{aligned}
a_1(\V{e}_{k} - \Vt{e}_{k}, \V{v}) + b(\delta_k-\delta_{k-1},\V{v}) & = 0, \qquad \forall \, \V{v} \in \V{X}, \\
b(q, \V{e}_{k}) & = 0, \qquad \forall \, q \in Q,
\end{aligned}
\end{equation}
\begin{equation}\label{eq:err3_Els}
a_1(\Vt{D}_{k},\V{v})
+ c^*(\V{e}_{k-1}, \V{z}, \V{v}) + c^*(\V{w}_{k-1}, \Vt{D}_k, \V{v})
+ b(r_{k-1},\V{v}) = -\numinus(\nabla \V{e}_{k-1},\nabla \V{v})\, \forall \V{v} \in \VX,
\end{equation}
\begin{equation}\label{eq:err4_Els}
\begin{aligned}
a_2(\V{D}_{k} - \Vt{D}_{k}, \V{S}) + b(r_k-r_{k-1},\V{S}) & = 0, \qquad \forall \, \V{S} \in \V{X}, \\
b(q, \V{D}_{k}) & = 0, \qquad \forall \, q \in Q.
\end{aligned}
\end{equation}
Then we test \eqref{eq:err1_Els} with $\V{v} = \Vt{e}_k \in \V{V}$ and \eqref{eq:err3_Els} with $\V{v} = \Vt{D}_k \in \V{V}$, respectively. Following \eqref{eq:err5}, we can rewrite $b(\cdot, \cdot)$ terms as 
\begin{equation}
    \label{eq:err5_Els}
    \begin{aligned}        
b(\delta_{k-1}, \Vt{e}_k - \V{e}_k) & = \frac{\|\nabla \delta_{k} \|^2_{\Luinvmap} - \|\nabla \delta_{k-1} \|^2_{\Luinvmap}}{2} +
\frac{\nuplus}{2} \left( \|\nabla \V{e}_{k} \|^2 - \|\nabla \Vt{e}_{k} \|^2 \right) + 
\frac{\gamma}{2} \left( \|\divergence \V{e}_{k} \|^2 - \|\divergence \Vt{e}_{k} \|^2\right), \\
b(\lambda_{k-1}, \Vt{D}_k - \V{D}_k) & = \frac{\|\nabla r_{k} \|^2_{\Lbinvmap} - \|\nabla r_{k-1} \|^2_{\Lbinvmap}}{2} + \frac{\nuplus}{2}\left( \|\nabla \V{D}_{k} \|^2 - \|\nabla \Vt{D}_{k} \|^2 \right) + 
\frac{\gamma}{2} \left( \|\divergence \V{D}_{k} \|^2 - \|\divergence \Vt{D}_{k} \|^2\right).
    \end{aligned}
\end{equation}
Now, picking $(\V{v},\V{S})=(\Vt{e}_k,\Vt{D}_k)$ in \eqref{eq:err1_Els} and \eqref{eq:err3_Els}, and using the skew-symmetry of the trilinear form $c^*(\cdot,\cdot,\cdot)$, we have
\begin{equation}\label{eq:err6_Els}
\begin{aligned}  
a_1(\Vt{e}_{k},\Vt{e}_{k})
+ c^*(\V{D}_{k-1}, \V{w}, \Vt{e}_{k})
+ b(\delta_{k-1},\Vt{e}_{k}-\V{e}_k) = 
- \numinus(\nabla \V{D}_{k-1},\nabla\Vt{e}_{k}), \\
a_1(\Vt{D}_{k},\Vt{D}_{k})
+  c^*(\V{e}_{k-1}, \V{z}, \Vt{D}_{k})
+  b(r_{k-1},\Vt{D}_{k}-\V{D}_{k}) = 
-\numinus (\nabla \V{e}_{k-1},\nabla \Vt{D}_{k}).
\end{aligned}
\end{equation}
Next we apply \eqref{eq:err5_Els} in \eqref{eq:err6_Els}, and use the standard bounds to get
\begin{equation}\label{eq:err7_Els}
\begin{aligned} 
\frac{\nuplus}{2} \left( \|\nabla \V{e}_{k} \|^2 + \|\nabla \Vt{e}_{k} \|^2 \right) & + 
\frac{\gamma}{2} \left( \|\divergence \V{e}_{k} \|^2 + \|\divergence \Vt{e}_{k} \|^2\right) + \frac{\|\nabla \delta_{k} \|^2_{\Luinvmap} - \|\nabla \delta_{k-1} \|^2_{\Luinvmap}}{2} \\ 
& \le \Mbd \|\nabla \V{w} \|\, \|\nabla \V{D}_{k-1} \|\, \|\nabla \Vt{e}_{k} \| + |\numinus|\, \|\nabla \V{D}_{k-1} \|\, \|\nabla \Vt{e}_{k} \| \\
& \le \left(\Mbd \Nbd_1 + |\numinus| \right) \|\nabla \V{D}_{k-1} \|\, \|\nabla \Vt{e}_{k} \|,
\end{aligned}
\end{equation}
and
\begin{equation}\label{eq:err8_Els}
\begin{aligned} 
\frac{\nuplus}{2} \left( \|\nabla \V{D}_{k} \|^2 + \|\nabla \Vt{D}_{k} \|^2 \right) & + 
\frac{\gamma}{2} \left( \|\divergence \V{D}_{k} \|^2 + \|\divergence \Vt{D}_{k} \|^2\right) + \frac{\|\nabla r_{k} \|^2_{\Lbinvmap} - \|\nabla r_{k-1} \|^2_{\Lbinvmap}}{2} \\ 
& \le \left(\Mbd \Nbd_1 + |\numinus| \right) \|\nabla \V{e}_{k-1} \|\, \|\nabla \Vt{D}_{k} \|.
\end{aligned}
\end{equation}
Adding the last two inequalities and applying Young's inequality produces
\begin{equation}\label{eq:err9_Els}   
\begin{aligned}
\frac{\nuplus}{2} \| (\nabla \V{e}_{k},\nabla \V{D}_{k}) \|^2 +  \frac{\nuplus}{4} \| (\nabla \Vt{e}_{k}, \nabla \Vt{D}_{k}) \|^2
& + \frac{\gamma}{2}\left( \|(\divergence \Vt{e}_{k},\divergence \Vt{D}_{k}) \|^2 + \|(\divergence \V{e}_{k},\divergence \V{D}_{k}) \|^2 \right) \\
& + \frac{ \|\nabla \delta_{k} \|^2_{\Luinvmap} +  \|\nabla r_{k} \|^2_{\Lbinvmap}}{2} \\
& \le \frac{ \|\nabla \delta_{k-1} \|^2_{\Luinvmap} +  \|\nabla r_{k-1} \|^2_{\Lbinvmap}}{2} + \nuplus \Lambda_3^2 \| (\nabla \V{e}_{k-1},\nabla \V{D}_{k-1}) \|^2 .
\end{aligned}
\end{equation}
Assuming that the small data condition \eqref{eq:smalldata_Els} holds, we can conclude that
\begin{equation}\label{eq:err10_Els}   
\begin{aligned}
\frac{\nuplus}{2} \| (\nabla \V{e}_{k},\nabla \V{D}_{k}) \|^2 +  \frac{\nuplus}{4} \| (\nabla \Vt{e}_{k}, \nabla \Vt{D}_{k}) \|^2
& + \frac{\gamma}{2}\left( \|(\divergence \Vt{e}_{k},\divergence \Vt{D}_{k}) \|^2 + \|(\divergence \V{e}_{k},\divergence \V{D}_{k}) \|^2 \right) \\
& + \frac{ \|\nabla \delta_{k} \|^2_{\Luinvmap} +  \|\nabla r_{k} \|^2_{\Lbinvmap}}{2} \\
& \le \frac{ \|\nabla \delta_{0} \|^2_{\Luinvmap} +  \|\nabla r_{0} \|^2_{\Lbinvmap}}{2} + \nuplus \Lambda_3^2 \| (\nabla \V{e}_{0},\nabla \V{D}_{0}) \|^2 \\
& = \frac{ \|\nabla \pplus \|^2_{\Luinvmap} +  \|\nabla \pminus \|^2_{\Lbinvmap}}{2} + \nuplus \Lambda_3^2 \| (\nabla \V{w},\nabla \V{z}) \|^2.
\end{aligned}
\end{equation}
The bound \eqref{eq:err10_Els} in turn implies uniform boundedness of the solution, as in \eqref{eq:err9}. Then, \eqref{eq:ErrConv_Els} is deduced easily from Lemma~\ref{lem:SeqConv0}. The geometric convergence of our approximation also follows easily from the inf-sup condition and Lemma \ref{ContractivitySequence}.
\end{proof}
\begin{remark}
In the classical formulation of the MHD system \eqref{eq:MHD_U}-\eqref{eq:MHD_Bmass}, the velocity $\V{u}$ and the magnetic field $\V{B}$ are coupled via two nonlinear terms in each equation. On the other hand, in the Els{\"a}sser reformulation \eqref{eq:Els1}-\eqref{eq:mass_Els2}, the coupling between $\V{w}$ and $\V{z}$ is through one nonlinear and one linear term. In theory, this allows for the construction of iterative methods with improved properties, assuming homogeneous Dirichlet boundary conditions for both $\V{u}$ and $\V{B}$. However, when the boundary conditions for these fields are not of the same type, determining the physically correct expressions for $\V{w}$ and $\V{z}$ becomes challenging.   
\end{remark}
\section{Numerical experiments}\label{sec:Numerics}
We present several numerical experiments to verify our theoretical results and to test the effectiveness of using FreeFem++ \cite{Hec2012} software. 
We consider the $\VX_h=\V{P}_2^d$,\, $Q_h=P_1$ finite element pairs in all computations, using direct solvers for 2D problems. 
The stopping criterion is taken as 
$$\max\left \lbrace \dfrac{\|\overrightarrow{X}_{p,k}-\overrightarrow{X}_{p,k-1}\|_{\ell_2}}{\| \overrightarrow{X}_{p,k} \|_{\ell_2}}, \dfrac{\| \overrightarrow{X}_{\lambda,k}-\overrightarrow{X}_{\lambda,k-1}\|_{\ell_2}}{\| \overrightarrow{X}_{\lambda,k} \|_{\ell_2}} \right \rbrace \leq 10^{-6}.$$ 
In some tests, we specify the nondimensional parameters ${\rm Re} = \dfrac{1}{\nu}$ and ${\rm Rm} = \dfrac{1}{\mu}$. Finally, in all runs, we set $\gamma=1$.
\subsection{Convergence tests}\label{sec:Convergence}
In a unit square $\Omega=(0,1)^2$, we consider the following manufactured solution of \eqref{eq:MHD_U}-\eqref{eq:MHD_Bmass}:
\begin{align*}
\V{u} &= \pi \left(\sin^2(\pi x) \sin(2\pi y), -\sin(2\pi x) \sin^2(\pi y) \right), \, p = \cos(\pi x) \cos(\pi y), \\
\V{B} &= \left(\sin(\pi x) \cos(\pi y), -\cos(\pi x) \sin(\pi y) \right), \, \lambda = \sin(\pi x) \sin(\pi y),
\end{align*}
where the source terms $\V{f}$ and $\V{g}$ are computed accordingly. Moreover, we took ${\rm Re}=1$, ${\rm Rm}=2$, and $\kappa=0.5$. The Tables \ref{tab:vel_mag_convergence}-\ref{tab:div_press_convergence} and \ref{tab:Els_convergence}-\ref{tab:Els_div_press_convergence} of errors and the corresponding convergence rates yield the expected results.
\begin{table}[htbp]
\centering
\caption{Part 1: Errors and rates for 2D manufactured solution of Algorithm \ref{algorithm:SIVS}}
\label{tab:vel_mag_convergence}
\resizebox{0.9\textwidth}{!}{
\begin{tabular}{|c|c|c|c|c|c|c|c|c|}
\toprule
$h$ & $\|\mathbf{u} - \mathbf{u}_h\|_{L^2}$ & Rate & $\|\mathbf{u} - \mathbf{u}_h\|_{H^1}$ & Rate & $\|\mathbf{B} - \mathbf{B}_h\|_{L^2}$ & Rate & $\|\mathbf{B} - \mathbf{B}_h\|_{H^1}$ & Rate \\
\hline
0.14  & 5.80e-03 & -    & 4.02e-01 & -    & 4.07e-04 & -    & 3.04e-02 & -    \\
\hline 
0.07  & 6.97e-04 & 3.06 & 1.02e-01 & 1.98 & 5.01e-05 & 3.02 & 7.63e-03 & 1.99 \\
\hline
0.035  & 8.61e-05 & 3.02 & 2.56e-02 & 1.99 & 6.24e-06 & 3.01 & 1.91e-03 & 2.00 \\
\hline
0.0177  & 1.07e-05 & 3.00 & 6.42e-03 & 2.00 & 7.79e-07 & 3.00 & 4.78e-04 & 2.00 \\
\hline
0.0088 & 1.34e-06 & 3.00 & 1.60e-03 & 2.00 & 9.73e-08 & 3.00 & 1.19e-04 & 2.00 \\
 \hline
\bottomrule
\end{tabular}%
}
\end{table}
\begin{table}[htbp]
\centering
\caption{Part 2: Errors and rates for 2D manufactured solution of Algorithm \ref{algorithm:SIVS}}
\label{tab:div_press_convergence}
\resizebox{0.9\textwidth}{!}{%
\begin{tabular}{|c|c|c|c|c|c|c|c|c|}
\toprule
$h$ & $\|\nabla \cdot \mathbf{u}_h\|_{L^2}$ & Rate & $\|\nabla \cdot \mathbf{B}_h\|_{L^2}$ & Rate & $\|p - p_h\|_{L^2}$ & Rate & $\|\lambda - \lambda_h\|_{L^2}$ & Rate \\
\hline
0.14  & 2.62e-01 & -    & 2.20e-02 & -    & 2.64e-02 & -    & 4.39e-03 & -    \\ \hline
0.07  & 6.89e-02 & 1.93 & 5.57e-03 & 1.98 & 2.36e-03 & 3.48 & 1.04e-03 & 2.08 \\ \hline
0.035  & 1.75e-02 & 1.98 & 1.40e-03 & 1.99 & 3.11e-04 & 2.92 & 2.58e-04 & 2.01 \\ \hline
0.0177  & 4.39e-03 & 1.99 & 3.50e-04 & 2.00 & 6.60e-05 & 2.24 & 6.43e-05 & 2.00 \\ \hline
0.0088 & 1.10e-03 & 2.00 & 8.74e-05 & 2.00 & 1.61e-05 & 2.03 & 1.61e-05 & 2.00 \\ \hline
\bottomrule
\end{tabular}%
}
\end{table}
\begin{table}[htbp]
\centering
\caption{Part 3: Errors and rates for 2D manufactured solution of Algorithm \ref{algorithm:SIVS_Els}}
\label{tab:Els_convergence}
\resizebox{0.9\textwidth}{!}{%
\begin{tabular}{|c|c|c|c|c|c|c|c|c|}
\hline
$h$ & $\|\mathbf{w} - \mathbf{w}_h\|_{L^2}$ & Rate & $\|\mathbf{w} - \mathbf{w}_h\|_{H^1}$ & Rate & $\|\mathbf{z} - \mathbf{z}_h\|_{L^2}$ & Rate & $\|\mathbf{z} - \mathbf{z}_h\|_{H^1}$ & Rate \\
\hline
0.14   & 5.94e-03 & -    & 4.12e-01 & -    & 5.67e-03 & -    & 3.92e-01 & -    \\
\hline
0.07  & 7.15e-04 & 3.05 & 1.05e-01 & 1.98 & 6.81e-04 & 3.06 & 9.96e-02 & 1.98 \\
\hline
0.035  & 8.84e-05 & 3.02 & 2.63e-02 & 1.99 & 8.41e-05 & 3.02 & 2.50e-02 & 1.99 \\
\hline
0.0177  & 1.10e-05 & 3.00 & 6.58e-03 & 2.00 & 1.05e-05 & 3.00 & 6.26e-03 & 2.00 \\
\hline
0.0088 & 1.43e-06 & 2.95 & 1.65e-03 & 2.00 & 1.37e-06 & 2.94 & 1.57e-03 & 2.00 \\
\hline
\end{tabular}
}
\end{table}
\begin{table}[htbp]
\centering
\caption{Part 4: Errors and rates for 2D manufactured solution of Algorithm \ref{algorithm:SIVS_Els}}
\label{tab:Els_div_press_convergence}
\resizebox{0.9\textwidth}{!}{%
\begin{tabular}{|c|c|c|c|c|c|c|c|c|}
\hline
$h$ & $\|\nabla \cdot \mathbf{w}_h\|_{L^2}$ & Rate & $\|\nabla \cdot \mathbf{z}_h\|_{L^2}$ & Rate & $\|p^+ - p^+_h\|_{L^2}$ & Rate & $\|p^- - p^-_h\|_{L^2}$ & Rate \\
\hline
0.14   & 2.70e-01 & -    & 2.55e-01 & -    & 2.69e-02 & -    & 2.62e-02 & -    \\
\hline
0.07  & 7.08e-02 & 1.93 & 6.72e-02 & 1.92 & 2.75e-03 & 3.29 & 2.15e-03 & 3.61 \\
\hline
0.035  & 1.80e-02 & 1.98 & 1.71e-02 & 1.98 & 4.73e-04 & 2.54 & 1.91e-04 & 3.49 \\
\hline
0.0177  & 4.50e-03 & 1.99 & 4.28e-03 & 1.99 & 1.11e-04 & 2.09 & 2.42e-05 & 2.98 \\
\hline
0.0088 & 1.13e-03 & 2.00 & 1.07e-03 & 2.00 & 2.76e-05 & 2.01 & 5.27e-06 & 2.20 \\
\hline
\end{tabular}%
}
\end{table}
\subsection{2D Hartmann flow}\label{subsec:2DHartmann}
Next, we test both of our Algorithms \ref{algorithm:SIVS} and \ref{algorithm:SIVS_Els} on a 2D Hartmann flow. The analytic solutions for the velocity, magnetic field, and pressure are
\begin{align*}
    \mathbf{u}(x,y) &= (u(y), 0), \,    \mathbf{B}(x,y) = (B(y), 1), \, p(x,y)= -0.1x - \kappa \frac{B^2(y)}{2},
\end{align*}
where the horizontal velocity and magnetic field profiles are given by:
\begin{align*}
    u(y) &= \frac{0.1 {\rm Re}}{{\rm Ha}  \cdot \tanh({\rm Ha} )} \left( 1 - \frac{\cosh(y {\rm Ha} )}{\cosh({\rm Ha} )} \right), \, 
    B(y) = \frac{0.1}{\kappa} \left( \frac{\sinh(y {\rm Ha} )}{\sinh({\rm Ha} )} - y \right),
\end{align*}
and the Hartmann number is defined as ${\rm Ha}  = \sqrt{ \kappa\,{\rm Re} {\rm Rm}}$. The problem is solved using Dirichlet boundary conditions for two different values of ${\rm Ha}$: ${\rm Ha} =1$ for $\kappa={\rm Re}={\rm Rm}=1$, and ${\rm Ha} =10$ for $\kappa={\rm Re}=10$, ${\rm Rm}=1$. The graphs of the exact and approximate solutions in Figure \ref{fig:2DHartmann} indicate excellent accuracy.
\begin{figure}
    \centering
    \includegraphics[width=0.9\linewidth]{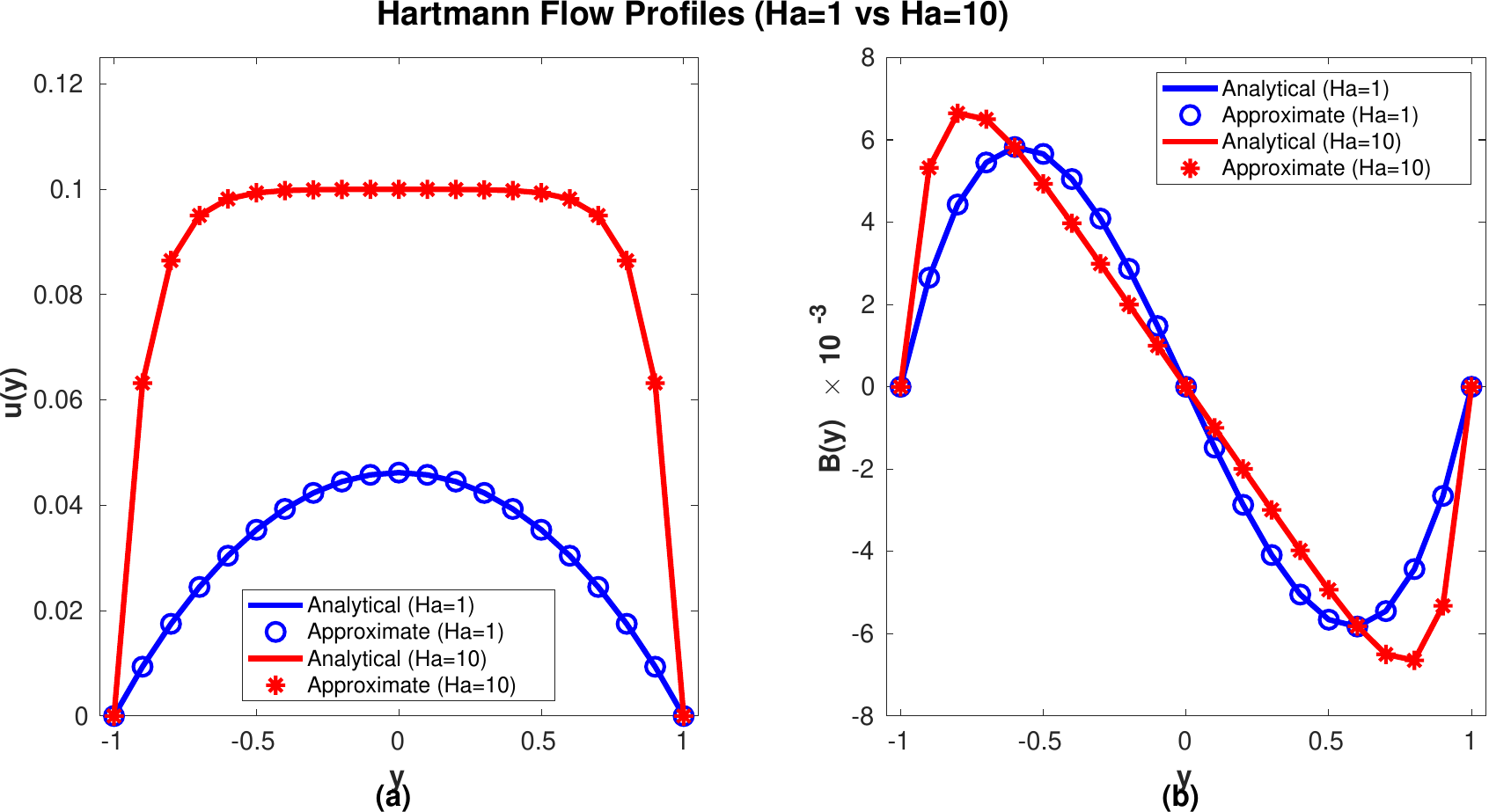}
    \caption{Graphs of $u_1$ versus $y$, and $B_1$ versus $y$ along $x=0.5$ line}
    \label{fig:2DHartmann}
\end{figure}
\subsection{3D Hartmann flow}\label{subsec:3DHartmann}
In this section, we test Algorithm \ref{algorithm:SIVS_Els} on a 3D Hartmann flow. The exact analytical solution for the in a rectangular duct $\Omega = [0, L] \times [-y_0, y_0] \times [-z_0, z_0]$, $y_0, z_0 \ll L$,  subjected to a constant transverse magnetic field $B_d = (0, 1, 0)$, is given as follows:
\begin{align*}
\V{u} &= (u(y,z), 0, 0), \\
\V{B} &= (B(y,z), 1, 0), \\
p &= -Gx - S_c \frac{B^2(y,z)}{2} + p_0.
\end{align*}
The velocity profile $u(y,z)$ and the induced magnetic field profile $B(y,z)$ are expressed using infinite Fourier series in the $z$-direction:
\begin{align*}
u(y,z) &= -\frac{1}{2} G R_e (z^2 - z_0^2) + \sum_{i=0}^{\infty} u_i(y) \cos(\lambda_i z), \\
B(y,z) &= \sum_{i=0}^{\infty} b_i(y) \cos(\lambda_i z).
\end{align*}
The Fourier coefficients $u_i(y)$ and $b_i(y)$ depend on the $y$-coordinate and are defined as:
\begin{align*}
u_i(y) &= A_i \cosh(p_1 y) + B_i \cosh(p_2 y) \\
b_i(y) &= \frac{1}{R_e S_c} \left( A_i \frac{\lambda_i^2 - p_1^2}{p_1} \sinh(p_1 y) + B_i \frac{\lambda_i^2 - p_2^2}{p_2} \sinh(p_2 y) \right)
\end{align*}
The eigenvalues $\lambda_i$ and the roots $p_1$ and $p_2$ governing the boundary layers are:
\begin{align*}
\lambda_i = \frac{(2i+1)\pi}{2z_0} \text{ and }
p_{1,2}^2 = \lambda_i^2 + \frac{Ha^2}{2} \pm Ha \sqrt{\lambda_i^2 + \frac{Ha^2}{4}}.
\end{align*}
$A_i$ and $B_i$ are determined using the boundary value $u_i(y_0)$ and the denominator $\gamma_i$:
\begin{align*}
u_i(y_0) &= \frac{-2 G R_e}{\lambda_i^3 z_0} \sin(\lambda_i z_0), \\
\gamma_i &= p_2(\lambda_i^2 - p_1^2) \sinh(p_1 y_0) \cosh(p_2 y_0) - p_1(\lambda_i^2 - p_2^2) \sinh(p_2 y_0) \cosh(p_1 y_0), \\
A_i &= \frac{-p_1(\lambda_i^2 - p_2^2)}{\gamma_i} u_i(y_0) \sinh(p_2 y_0), \\
B_i &= \frac{p_2(\lambda_i^2 - p_1^2)}{\gamma_i} u_i(y_0) \sinh(p_1 y_0).
\end{align*}
The problem is solved using Dirichlet boundary conditions for ${\rm Ha} =0.01$ corresponding to $\kappa={\rm Re}=1$, ${\rm Rm}=1e-4$. The graphs of the exact and approximate solutions in Figure \ref{fig:3DHartmann} show very good accuracy.
\begin{figure}
    \centering
    \includegraphics[width=0.8\linewidth]{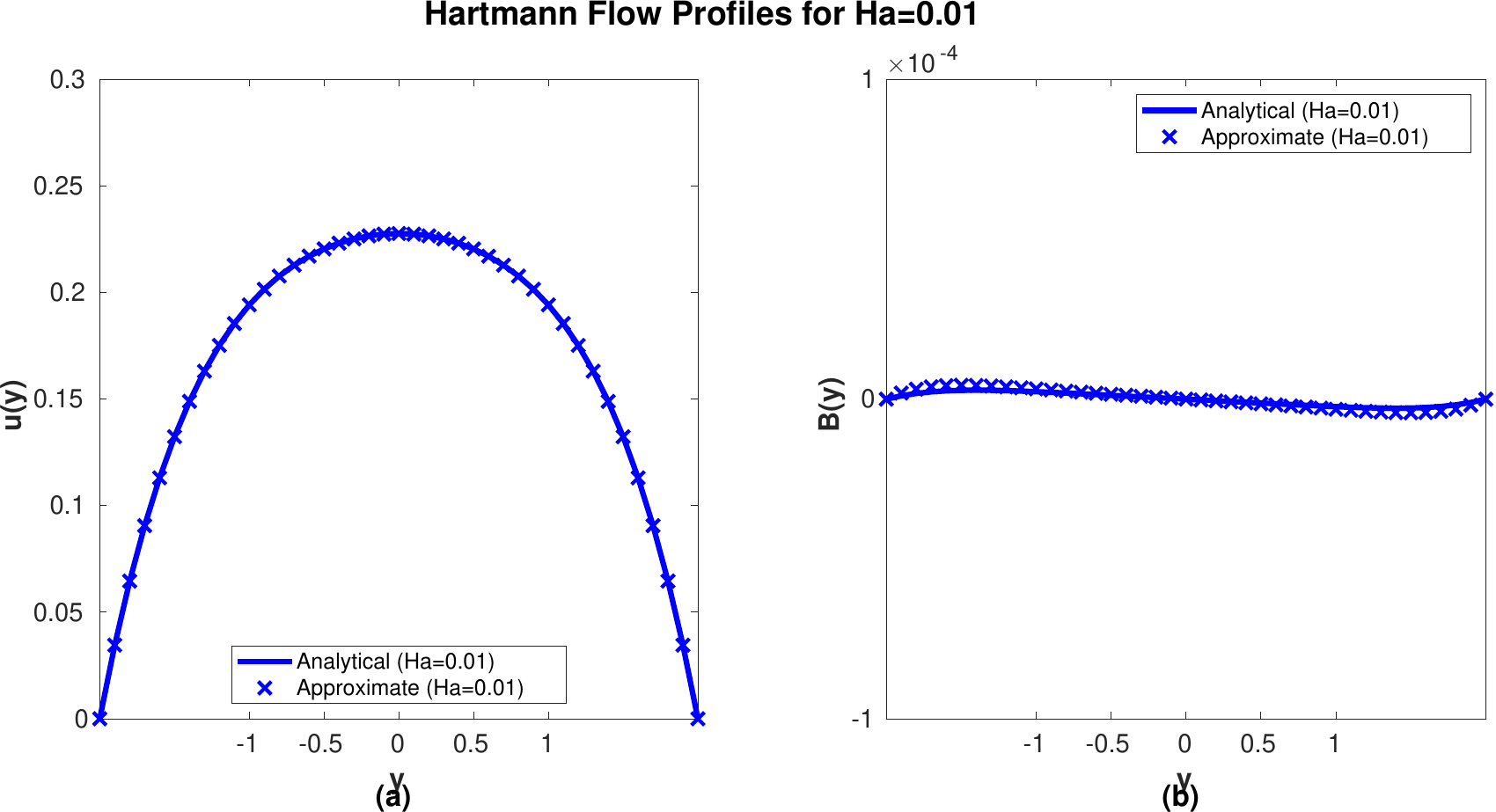}
    \caption{Graphs of $u_1$ versus $y$, and $B_1$ versus $y$ along $x=0.5$ line}
    \label{fig:3DHartmann}
\end{figure}
\subsection{2D MHD lid driven cavity flow}\label{subsec:2DCavity}
In this subsection, we test our Algorithms on a well-known 2D lid-driven cavity flow problem. The computational domain is $\Omega = (0,1)^2$, where the top lid moves in the positive $x$ direction at unit speed. The boundary conditions are taken to be no-slip along the remaining walls. To avoid solution irregularities at the upper corners, we consider a regularized initial data at the upper boundary, as in \cite{Frutos2016}.

We tested various values of ${\rm Re}, {\rm Rm}$ and the coupling number $\kappa$. The mesh is uniform of size $160 \times 160$. The runs with the smallest values of the parameters were initiated with zero, while the runs at higher values are initiated from the converged solution of the previous simulations.

The velocity streamlines superimposed on speed contours are shown in Figures \ref{fig:Cavity1}-\ref{fig:Cavity2}, which perfectly match the reference results \cite{XU2022105372}, except for the velocity streamlines at ${\rm Re}={\rm Rm}=1$ and $\kappa=5000$ case in Figure \ref{fig:Cavity2}. We ran this case with two different continuations. The first case was run with $\kappa=1,50,100,200,1000,2000,5000$, and 
the second case with $\kappa=1,50,100,200,1000,2000,3000,3500,4000,5000$, both giving the same solution.
Representative magnetic field lines are given in Figure \ref{fig:Cavity3}.
\begin{figure}[t!]
\centering
\begin{subfigure}[t]{0.33\textwidth}
  \centering
\includegraphics[scale=0.24]{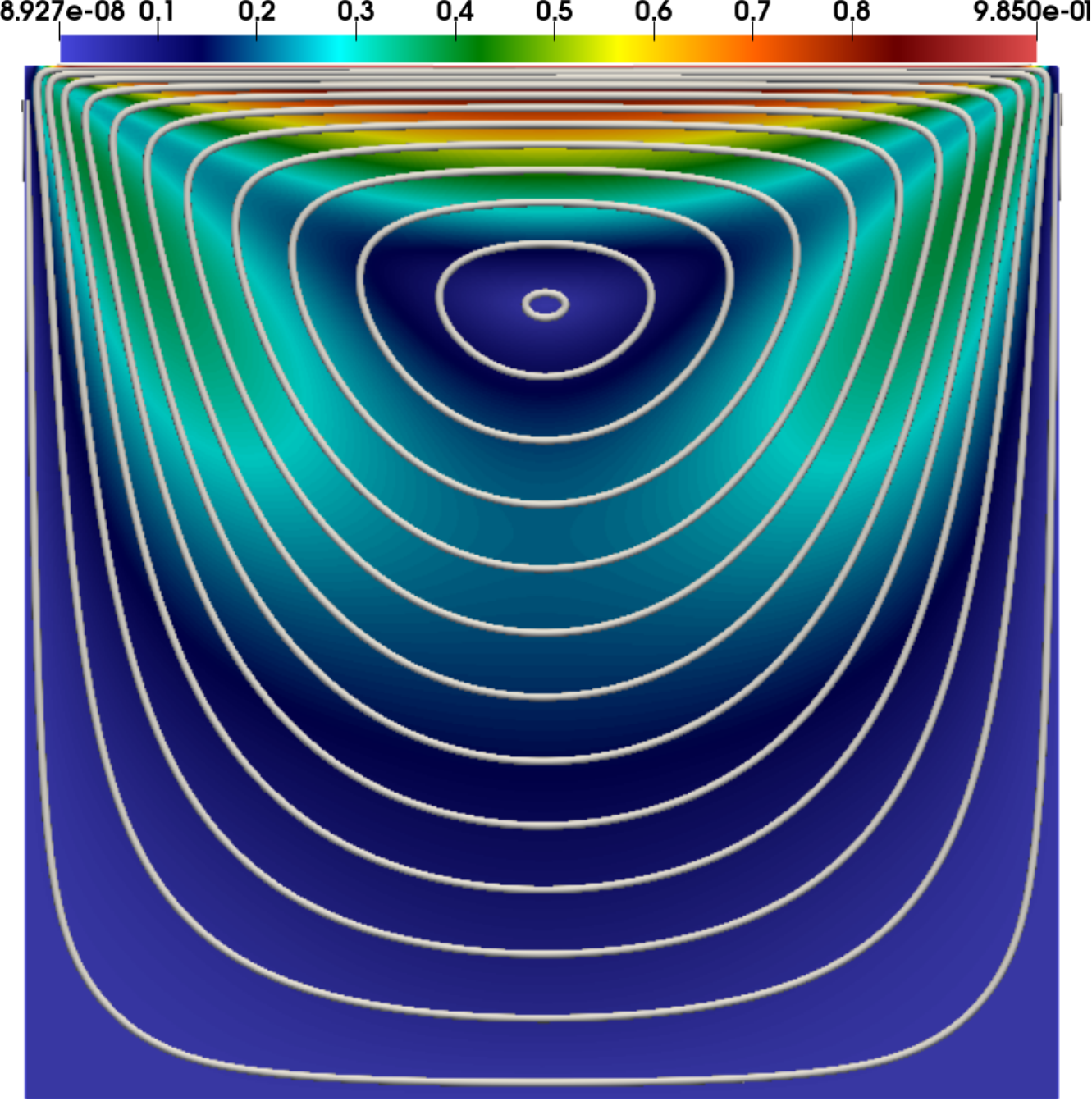}
\end{subfigure}%
\begin{subfigure}[t]{0.33\textwidth}
  \centering
\includegraphics[scale=0.24]{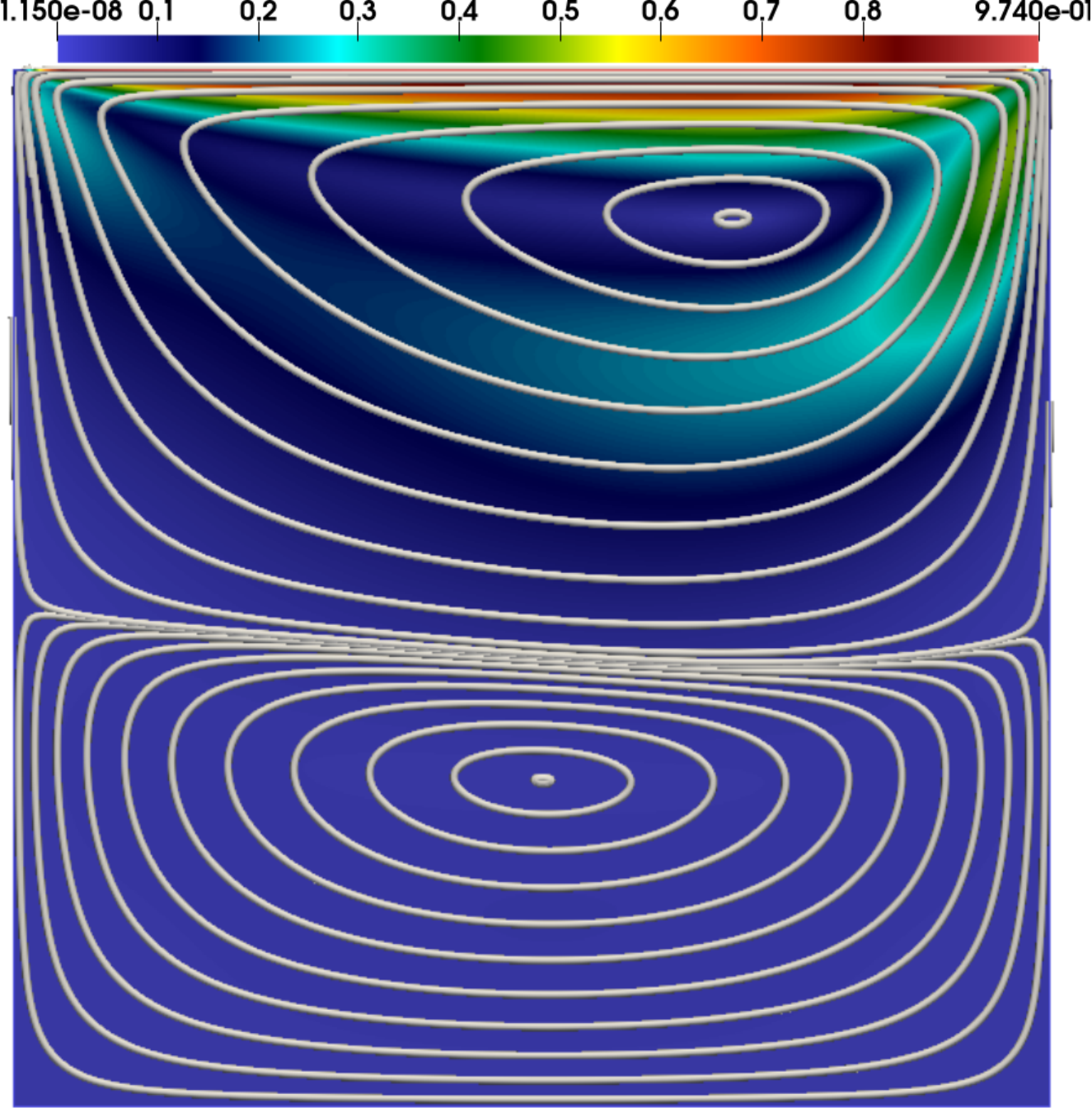}
\end{subfigure}%
\begin{subfigure}[t]{0.33\textwidth}
  \centering
\includegraphics[scale=0.24]{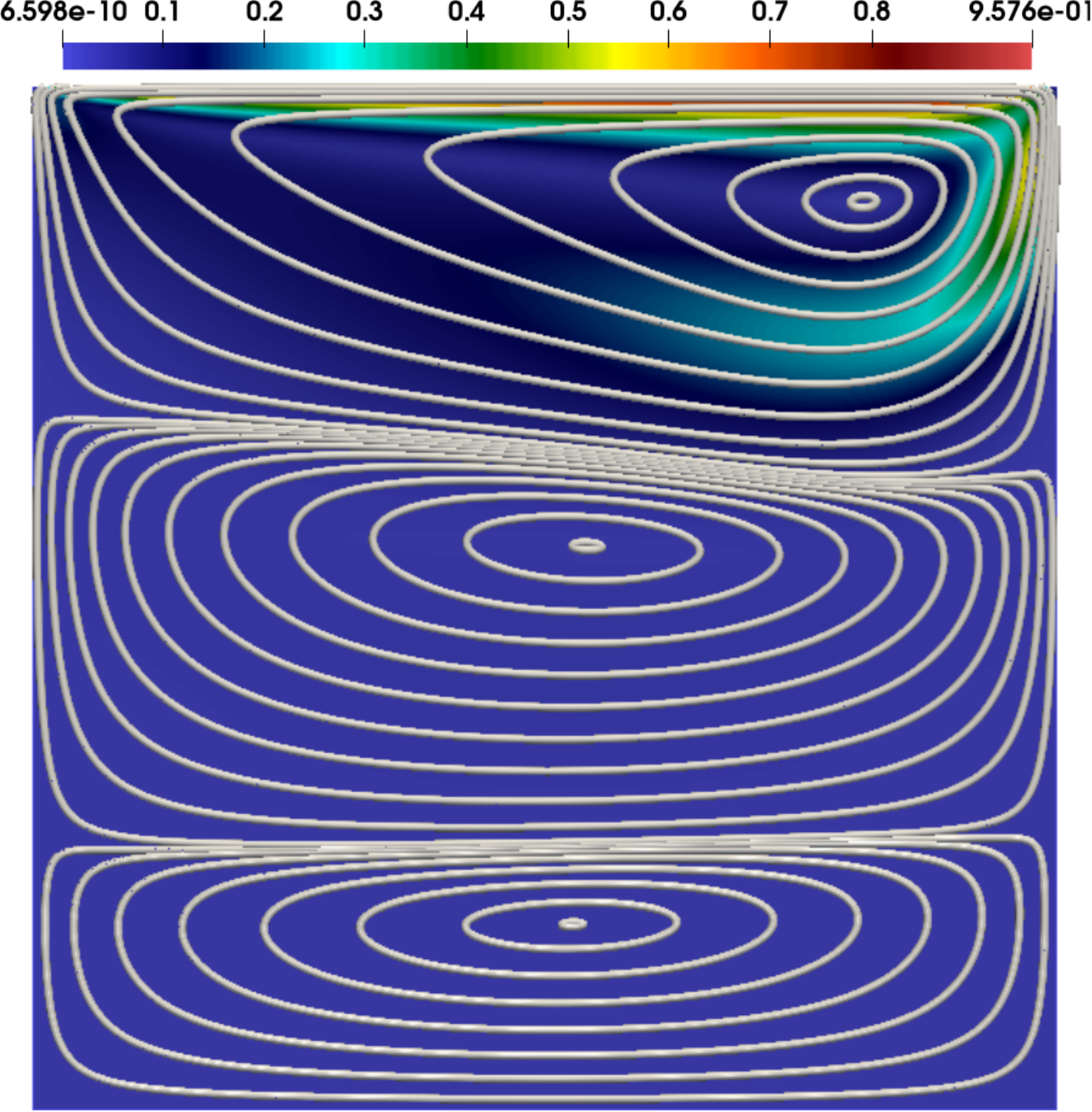}
\end{subfigure}
\caption{Velocity streamlines with ${\rm Rm}=1$ \& $\kappa=1$: ${\rm Re}=1$ (left), ${\rm Re}=100$ (middle), \& ${\rm Re}=500$ (right)}
\label{fig:Cavity1}
\end{figure}
\begin{figure}[ht!]
\centering
\begin{subfigure}[t]{0.33\textwidth}
  \centering
\includegraphics[scale=0.24]{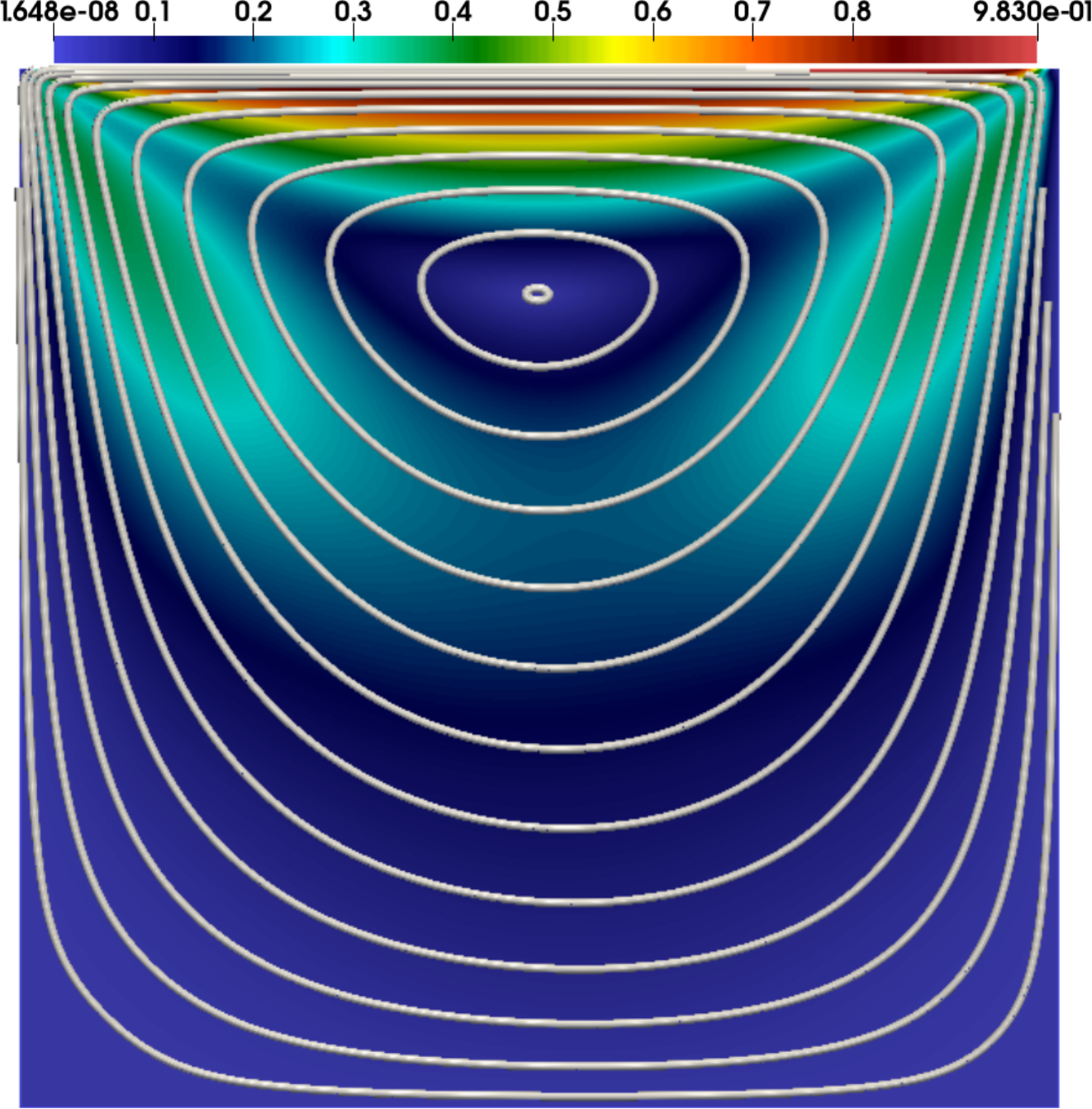}
\end{subfigure}%
\begin{subfigure}[t]{0.33\textwidth}
  \centering
\includegraphics[scale=0.24]{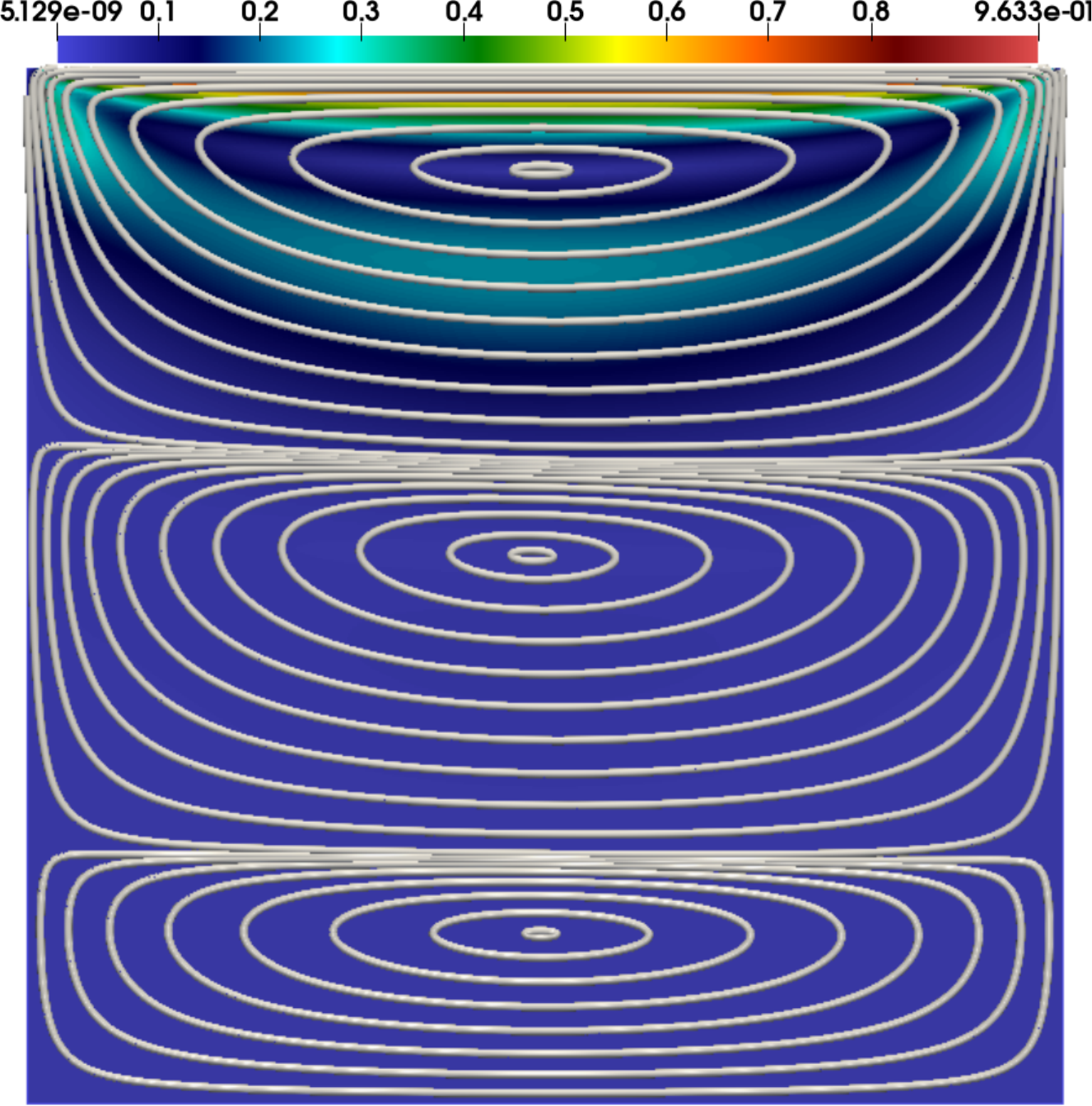}
\end{subfigure}%
\begin{subfigure}[t]{0.33\textwidth}
  \centering
\includegraphics[scale=0.24]{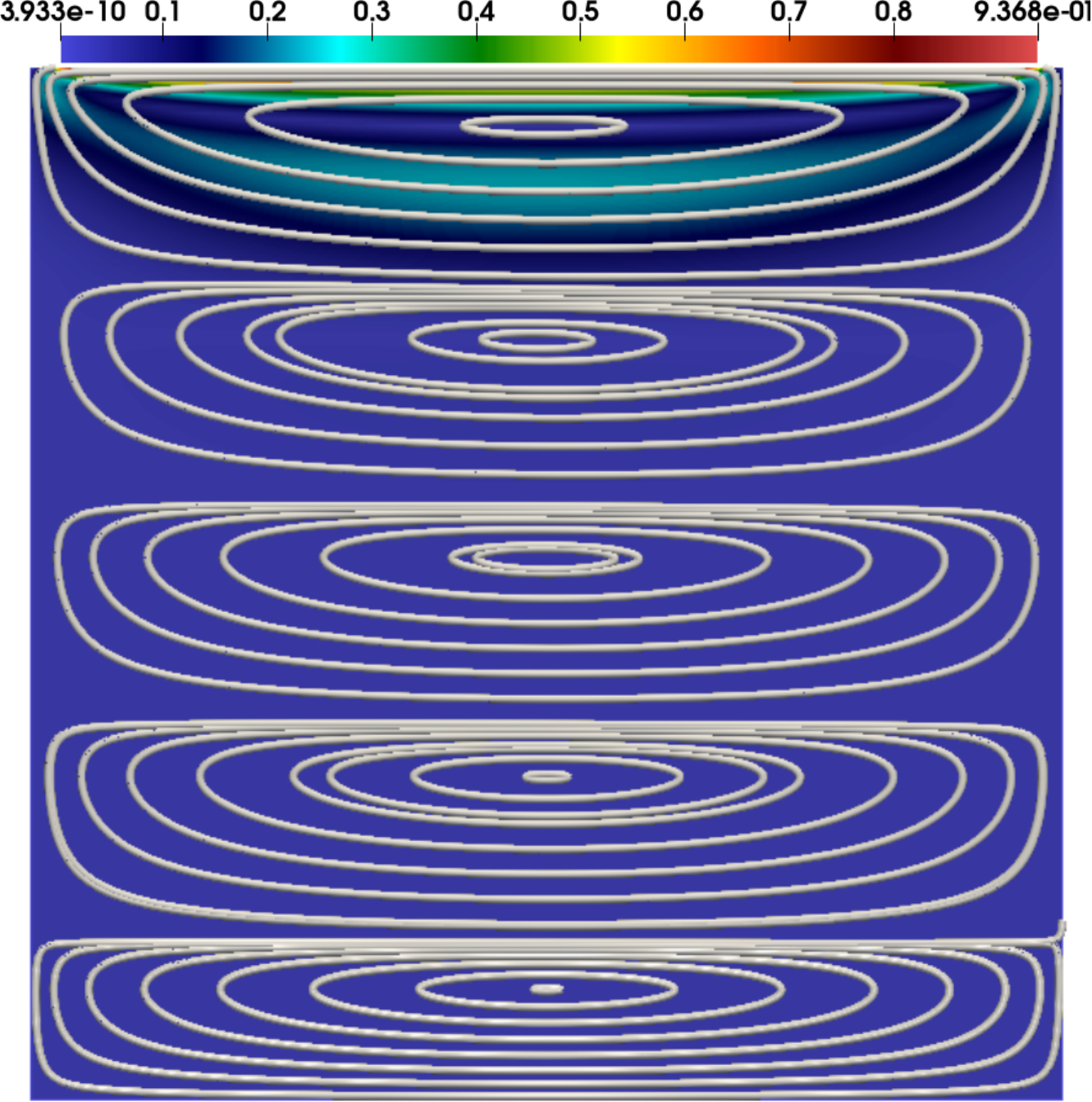}
\end{subfigure}
\caption{Velocity streamlines: ${\rm Re}=\kappa=1$, ${\rm Rm}=10$ (left);  ${\rm Re}={\rm Rm}=1$ \& $\kappa=500$ (middle), \& $\kappa=5000$ (right)}
\label{fig:Cavity2}
\end{figure}
\begin{figure}[ht!]
\centering
\begin{subfigure}[t]{0.33\textwidth}
  \centering
\includegraphics[scale=0.24]{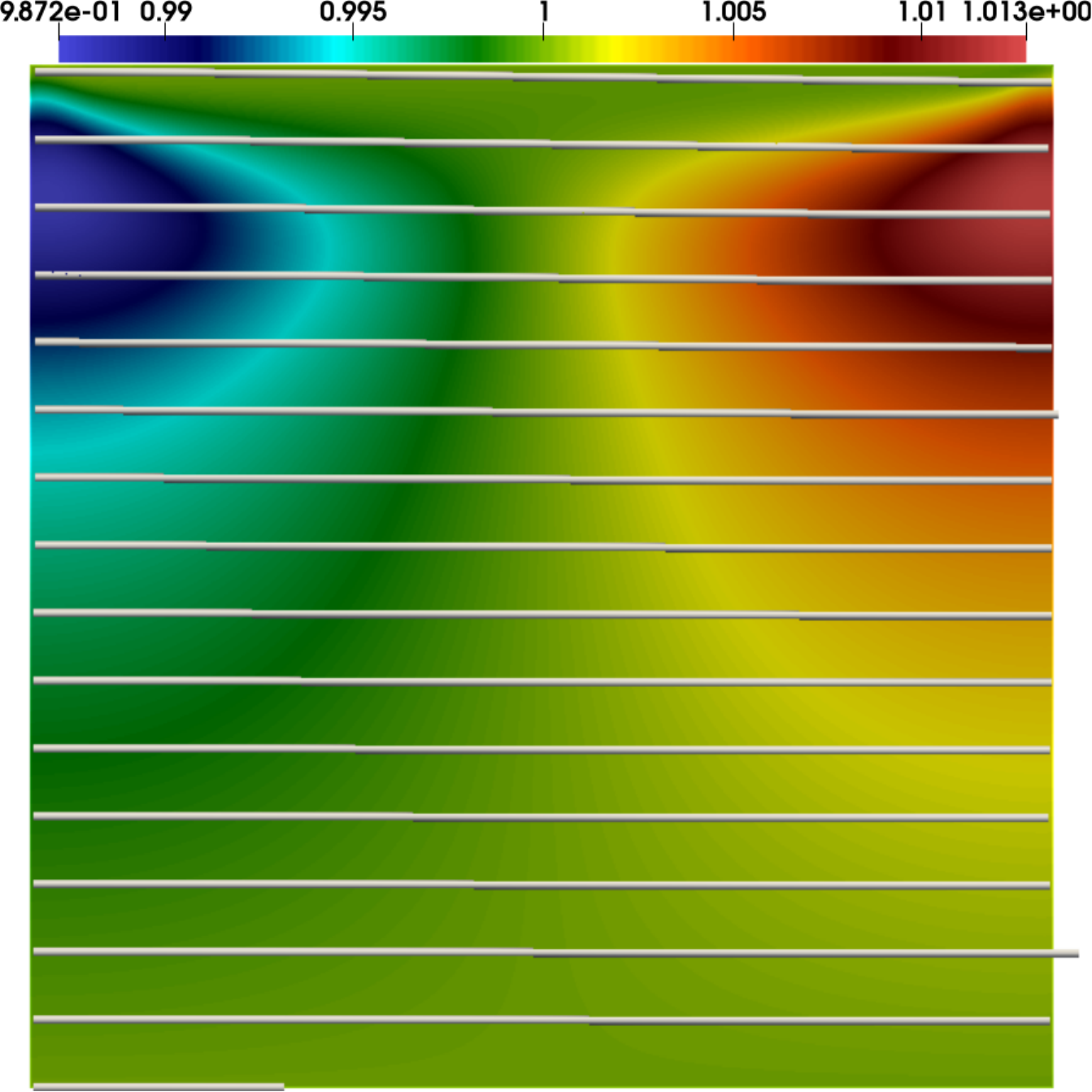}
\end{subfigure}%
\begin{subfigure}[t]{0.33\textwidth}
  \centering
\includegraphics[scale=0.24]{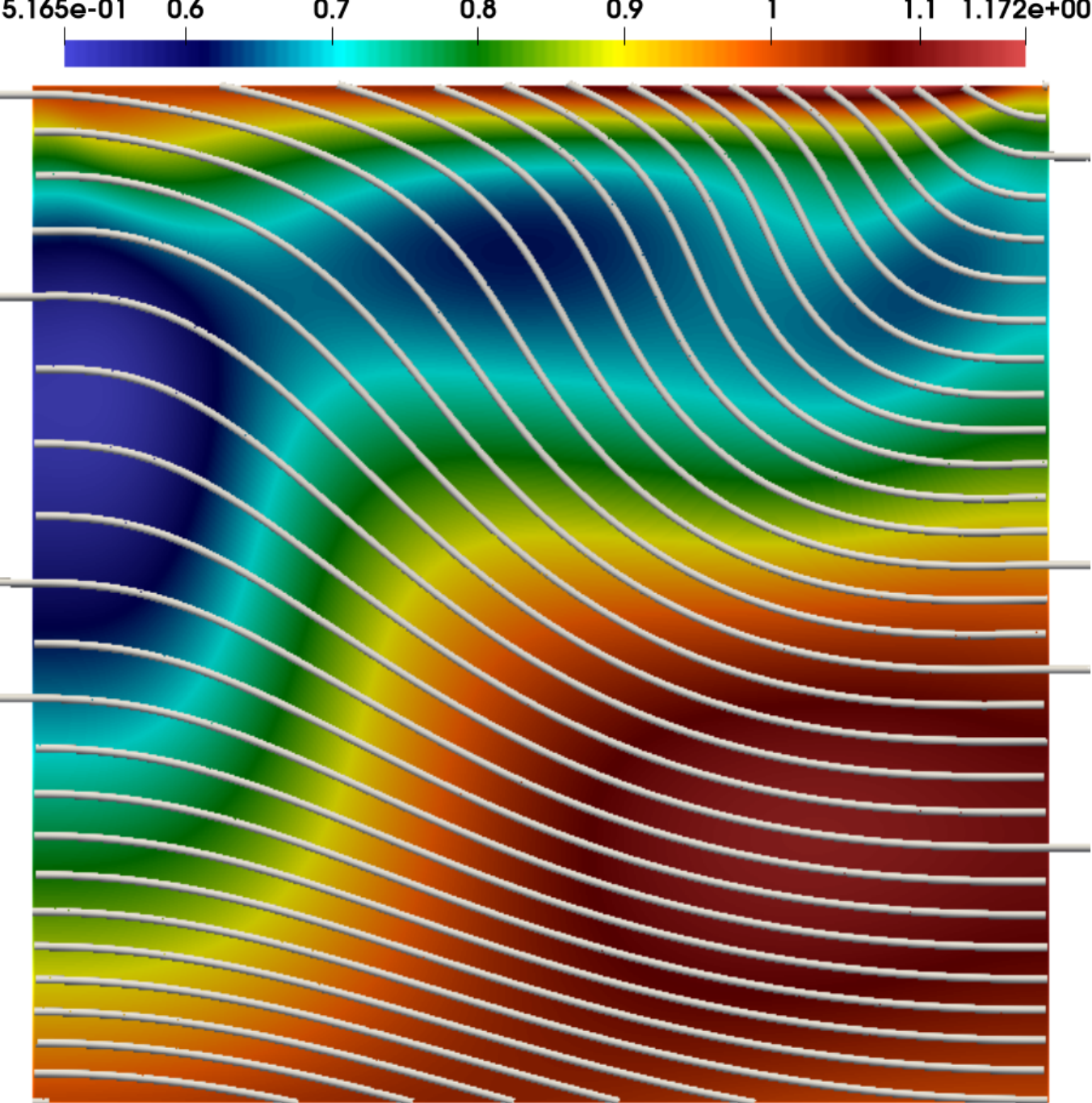}
\end{subfigure}%
\begin{subfigure}[t]{0.33\textwidth}
  \centering
\includegraphics[scale=0.24]{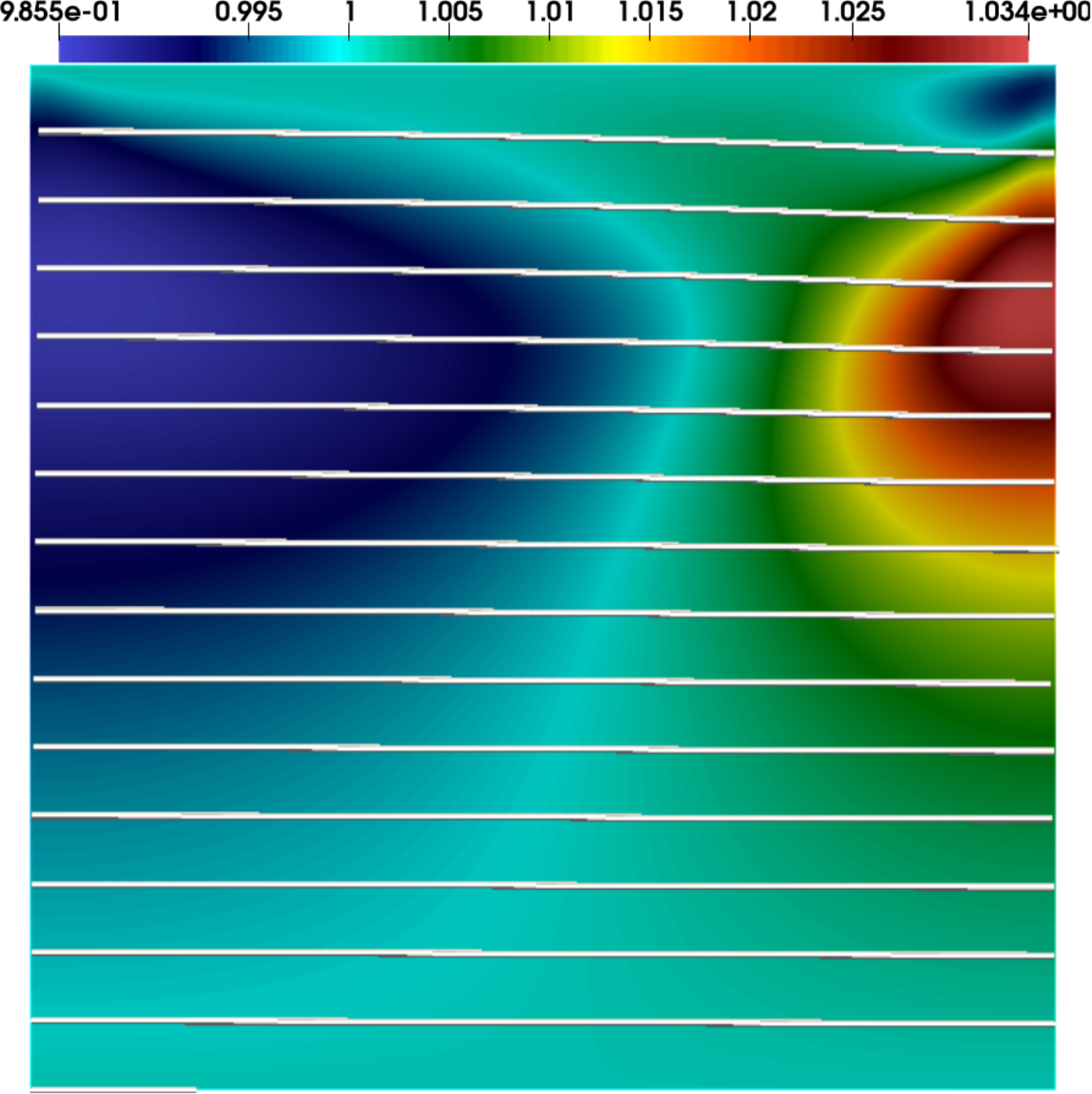}
\end{subfigure}
\caption{Magnetic field lines: ${\rm Re}={\rm Rm}=1$ \& $\kappa=5000$ (left);  ${\rm Re}=\kappa=1$ \& ${\rm Rm}=10$ (middle); ${\rm Re}=500$ \& ${\rm Rm}=\kappa=1$ (right)}
\label{fig:Cavity3}
\end{figure}
\subsection{2D MHD flow over a step}\label{subsec:2DStep}
This example is a classical flow over a step under a transverse magnetic field. The domain is $\Omega=(-0.25,0.75) \times (-0.125,0.125) \backslash (-0.25,0] \times (-0.125,0]$. 

The following boundary conditions are imposed:
\begin{equation}\label{eq:2DStep_MHD_BCs}
	\begin{array}{rcl}
	\V{u} &=& \left(-25.6(y-0.125)y\, ,\, 0\right) \quad \text{ on } x=-0.25, \\
    -\nu (\V{n} \cdot \nabla)\V{u} + p\V{n} &=& \V{0} \quad \text{ on }x=0.75, \\
	\V{u} &=& \V{0} \quad \text{ elsewhere}, \\
	\V{B} \times \V{n} &=& \V{B}_D \times \V{n} \quad \text{ on } \partial \Omega, \\
    \dfrac{\partial \V{B}}{\partial \V{n}} \cdot \V{n} & = & 0 \quad \text{ on } \partial \Omega,
	\end{array}
	\end{equation}
where $\V{B}_D = (1,0)^T$. \AT{We note that the last condition in \eqref{eq:2DStep_MHD_BCs} must be enforced in the Laplacian formulation of \eqref{eq:MHD_B} to ensure stability.}  Body forces are zero. The problem parameters are chosen as $\nu=10^{-2}$, $\mu=10^5$, with $\kappa=2.5\cdot 10^4$ and $\kappa=10^5$. The finite element mesh consisting of $229752$ triangles is used.

As expected from the reference solution, the streamlines are correctly captured, and pressure values drop past the step, cf. Fig. \ref{fig:2DStep1}. Moreover, we observe that the corner vertex gets damped more for an increasing value of the coupling parameter $\kappa$, cf. Fig. \ref{fig:2DStep2}.
\begin{figure}[t!]
\centering
\includegraphics[scale=0.3]{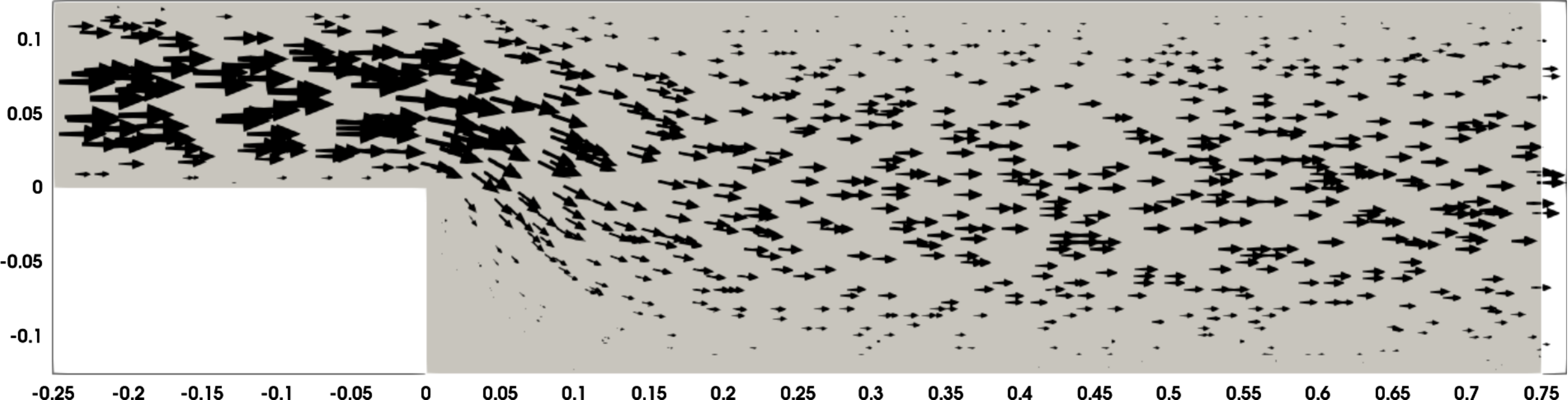} \\
\includegraphics[scale=0.3]{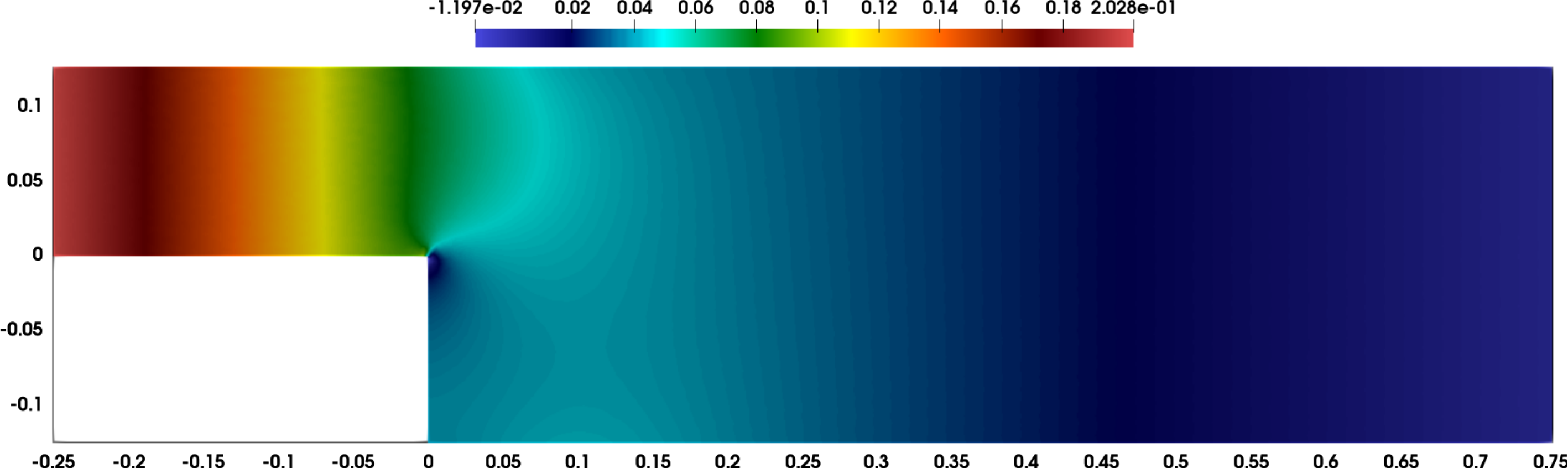}
\caption{Velocity field (top) and Pressure contours (bottom) for 2D flow over a step at $\kappa=2.5 \cdot 10^4$}
\label{fig:2DStep1}
\end{figure}
\begin{figure}[t!]
\centering
\includegraphics[scale=0.25,trim=6in 0in 0in 5in, clip]{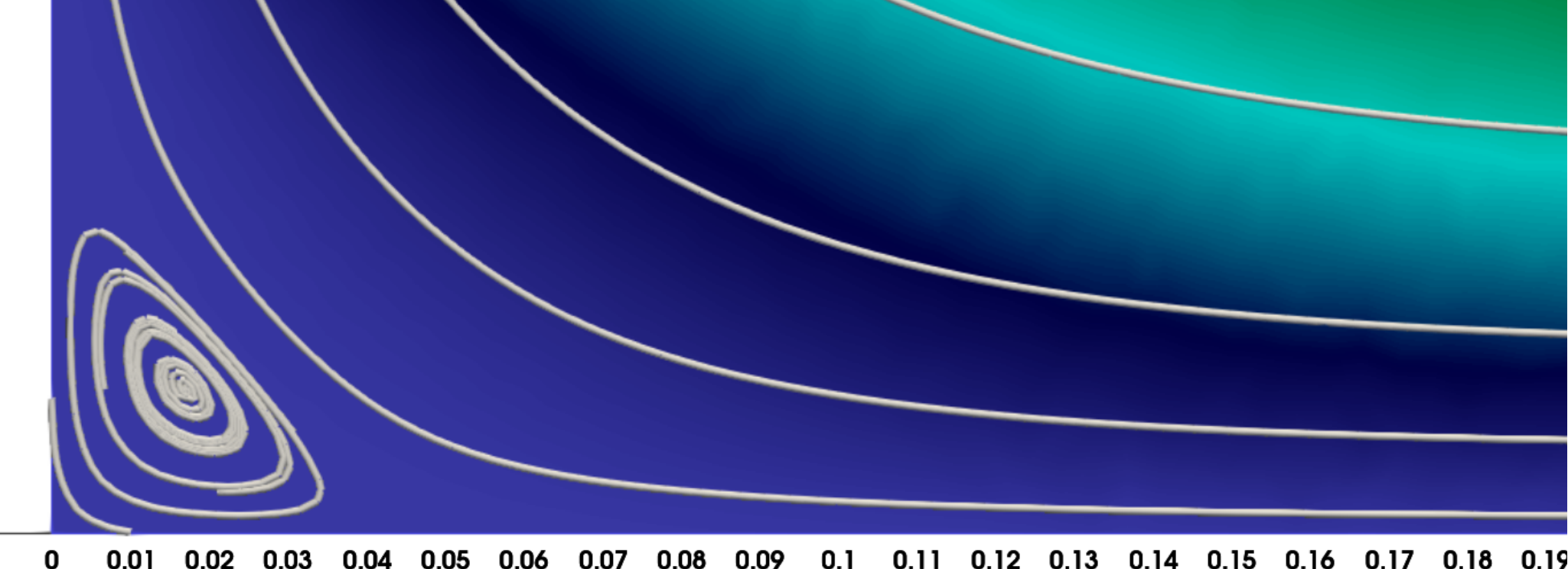} \\
\,\includegraphics[scale=0.25,trim=6.1in 0in 0in 4.9in, clip]{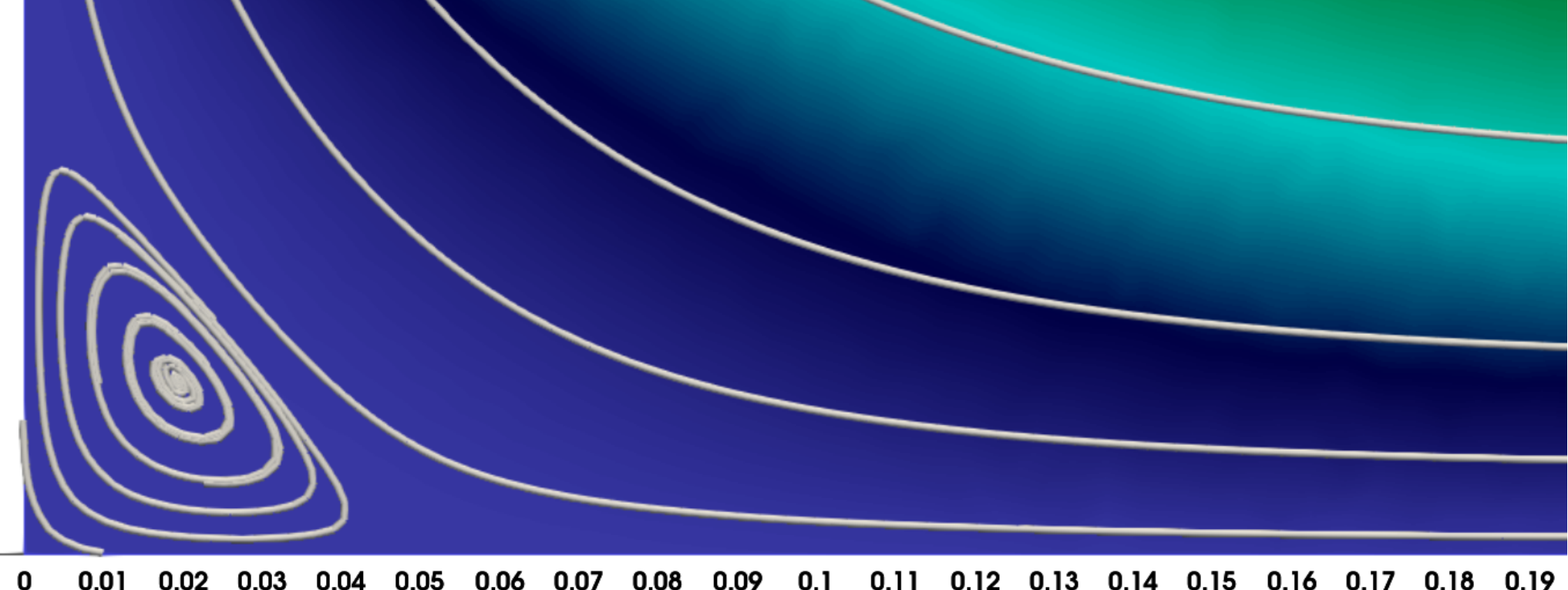}
\caption{Velocity streamlines zoomed in past the step: $\kappa=10^5$ (top) and $\kappa=2.5 \cdot 10^4$ (bottom)}
\label{fig:2DStep2}
\end{figure}
\section{Conclusion}\label{sec:Conclusion}
In this work, we first established the existence and uniqueness of the solution to the stationary MHD problem under suitable assumptions. We then analyzed the classical Picard iterative scheme and proved its convergence in the three-dimensional setting under the same conditions ensuring uniqueness of the continuous problem.

To overcome some implementation difficulties related to the strong coupling of the unknowns, we introduced a new iterative algorithm inspired by the IVS strategy. This approach allows for a decoupling of the main variables, leading to simpler and more efficient linear subproblems. We proved that this new iterative method converges toward the unique solution of the original MHD problem.

For each proposed algorithm, we explicitly derived and detailed the associated linear systems, which provides a practical framework for numerical implementation.

Furthermore, we extended our analysis to the Elsässer formulation of the MHD equations, showing that the proposed methodology can also be adapted to this alternative formulation.

Finally, several two-dimensional and three-dimensional numerical experiments were performed to validate the theoretical results and illustrate the efficiency and robustness of the proposed methods.

A natural continuation of this work concerns the unsteady MHD problem. Future investigations will focus on the extension of the proposed strategies to the time-dependent setting, with particular attention devoted to the analysis of temporal and spatial stability, as well as convergence properties of the fully discrete schemes.

\bibliographystyle{abbrv}
\bibliography{references}
\end{document}